\documentclass[twoside,11pt,reqno]{amsart}
\usepackage{amsmath,amssymb,amscd,mathrsfs,epic,wasysym,latexsym,tikz,mathrsfs,cite,hyperref,enumerate}
\usepackage{stmaryrd}
\usepackage{pb-diagram}
\usepackage[matrix,arrow]{xy}
\usepackage{commath}
\usepackage{tikz-cd}
\usepackage{upgreek}
\usepackage{ytableau}

\DeclareFontFamily{OT1}{pzc}{}
\DeclareFontShape{OT1}{pzc}{m}{it}{<-> s * [1.10] pzcmi7t}{}
\DeclareMathAlphabet{\mathpzc}{OT1}{pzc}{m}{it}

\makeatletter

\numberwithin{equation}{section}

\newtheorem{Proposition}[equation]{Proposition}
\newtheorem{Lemma}[equation]{Lemma}
\newtheorem{Theorem}[equation]{Theorem}
\newtheorem{Corollary}[equation]{Corollary}
\newtheorem{Inductive Assumption}[equation]{Inductive Assumption}
\theoremstyle{definition}  

\newtheorem{Example}[equation]{Example}

\let\<\langle
\let\>\rangle

\newcommand\Comment[2][\relax]{\space\par\medskip\noindent%
   \fbox{\begin{minipage}{\textwidth}\textbf{Comment\ifx\relax#1\else---#1\fi}\newline%
        #2\end{minipage}}\medskip
}

\def\bi{\text{\boldmath$i$}}
\def\bj{\text{\boldmath$j$}}

\def\bk{\text{\boldmath$k$}}

\def\b1{\text{\boldmath$1$}}

\def\pmod#1{\text{ }(\text{\rm mod } #1)\,}

\newcommand{\CT}{\sf{CT}}

\def\cont{{\operatorname{cont}}}
\def\core{{\operatorname{core}}}

\newcommand{\Z}{\mathbb{Z}}
\newcommand{\N}{\mathbb{N}}

\def\phi{{\varphi}}

\newcommand{\di}{{\operatorname{div}}}

\newcommand{\Fock}{{\mathscr F}}

\newcommand{\quot}{\operatorname{quot}}

\newcommand{\la}{\lambda}
\newcommand{\La}{\Lambda}
\newcommand{\al}{\alpha}

\def\Si{\mathfrak{S}}
\newcommand{\si}{\sigma}
\newcommand{\om}{\omega}
\newcommand{\Om}{\Omega}

\newcommand{\de}{\delta}

\newcommand{\ka}{\kappa}

\newcommand{\pzs}{{\mathsf{s}}}
\newcommand{\pzp}{{\mathsf{p}}}
\newcommand{\pzm}{{\mathsf{m}}}
\newcommand{\pze}{{\mathsf{e}}}
\newcommand{\pzh}{{\mathsf{h}}}
\newcommand{\pzq}{{\mathsf{q}}}
\newcommand{\pzQ}{{\mathsf{Q}}}
\newcommand{\pzP}{{\mathsf{P}}}
\newcommand{\pzf}{{\mathsf{f}}}

\newcommand{\Res}{{\operatorname{Res}}}

\newcommand{\C}{{\mathbb C}}
\newcommand{\Q}{{\mathbb Q}}

\newcommand{\Sym}{{\mathsf{Sym}}}

\newcommand{\ttA}{{\tt A}}
\newcommand{\ttB}{{\tt B}}
\newcommand{\ttC}{{\tt C}}

\newcommand{\sfp}{{\mathpzc{p}}}

\def\ttB{{\mathtt B}}

\def\col{{\operatorname{col}}}

\newcommand{\Add}{\operatorname{Ad}}
\newcommand{\Rem}{\operatorname{Re}}
\newcommand{\PA}{\operatorname{PAd}}
\newcommand{\PR}{\operatorname{PRe}}

\def\g{{\mathfrak g}}

\def\Par{{\mathscr P}}

\def\ula{{\underline{\lambda}}}

\def\ual{{\underline{\al}}}

\def\b{\mathfrak{b}}
\def\k{\Bbbk}

\def\spa{\operatorname{span}}
\def\height{{\operatorname{ht}}}

\def\wt{{\operatorname{wt}}}

\def\iso{\stackrel{\sim}{\longrightarrow}}

\def\T{{\sf T}}

\def\lan{\langle}
\def\ran{\rangle}

\def\ttB{{\mathtt B}}

{\catcode`\|=\active
  \gdef\set#1{\mathinner{\lbrace\,{\mathcode`\|"8000%
  \let|\midvert #1}\,\rbrace}}
}
\def\midvert{\egroup\mid\bgroup}

\colorlet{darkgreen}{green!50!black}
\tikzset{dots/.style={very thick,loosely dotted},
         greendot/.style={fill,circle,color=darkgreen,inner sep=1.5pt,outer sep=0},
         blackdot/.style={fill,circle,color=black,inner sep=1.1pt,outer sep=0},
         graydot/.style={fill,circle,color=gray,inner sep=1.1pt,outer sep=0},
         reddot/.style={fill,circle,color=red,inner sep=1.1pt,outer sep=0},
         bluedot/.style={fill,circle,color=blue,inner sep=1.1pt,outer sep=0}
}
\def\greendot(#1,#2){\node[greendot] at(#1,#2){}}
\def\blackdot(#1,#2){\node[blackdot] at(#1,#2){}}
\def\graydot(#1,#2){\node[graydot] at(#1,#2){}}
\def\reddot(#1,#2){\node[reddot] at(#1,#2){}}
\def\bluedot(#1,#2){\node[bluedot] at(#1,#2){}}

\newenvironment{braid}{
  \begin{tikzpicture}[baseline=6mm,black,line width=.7pt, scale=0.32,
                      draw/.append style={rounded corners},
                      every node/.append style={font=\fontsize{5}{5}\selectfont}]%
  }{\end{tikzpicture}
}

\def\Grid(#1,#2){
  \draw[very thin,gray,step=2mm] (0,0)grid(#1,#2);
  \draw[very thin,darkgreen,step=10mm] (0,0)grid(#1,#2);
}

\newcommand\Tableau[2][\relax]{
  \begin{tikzpicture}[scale=0.5,draw/.append style={thick,black}]
    \ifx\relax#1\relax%
    \else 
      \foreach\box in {#1} { \filldraw[blue!30]\box+(-.5,-.5)rectangle++(.5,.5); }
    \fi
    \newcount\row\newcount\col
    \row=0
    \foreach \Row in {#2} {
       \col=1
       \foreach\k in \Row {
          \draw(\the\col,\the\row)+(-.5,-.5)rectangle++(.5,.5);
          \draw(\the\col,\the\row)node{\k};
          \global\advance\col by 1
       }
       \global\advance\row by -1
    }
  \end{tikzpicture}
}

\newcommand\YoungDiagram[2][\relax]{
  \begin{tikzpicture}[scale=0.5,draw/.append style={thick,black}]
    \ifx\relax#1\relax%
    \else 
    \foreach\box in {#1} {
      \filldraw[blue!30]\box rectangle ++(1,1);
    }
    \fi
    \newcount\row
    \row=0
    \foreach \col in {#2} {
       \draw(1,\the\row)grid ++(\col,1);
       \global\advance\row by -1
    }
  \end{tikzpicture}
}

\newenvironment{Young}{\begingroup
       \def\vr{\vrule height0.89\hoogte width\dikte depth 0.2\hoogte}
       \def\fbox##1{\vbox{\offinterlineskip
                    \hrule height\dikte
                    \hbox to \breedte{\vr\hfill##1\hfill\vr}
                    \hrule height\dikte}}
       \vbox\bgroup \offinterlineskip \tabskip=-\dikte \lineskip=-\dikte
            \halign\bgroup &\fbox{##\unskip}\unskip  \crcr }
       {\egroup\egroup\endgroup}
\def\diagram#1{\relax\ifmmode\vcenter{\,\begin{Young}#1\end{Young}\,}\else%
              $\vcenter{\,\begin{Young}#1\end{Young}\,}$\fi}

\begin{document}

\title[Dimensions of RoCK blocks of quiver Hecke superalgebras]{{On dimensions of RoCK blocks of cyclotomic quiver Hecke superalgebras}}

\author{\sc Alexander Kleshchev}
\address{Department of Mathematics\\ University of Oregon\\ Eugene\\ OR 97403, USA}
\email{klesh@uoregon.edu}

\dedicatory{To the memory of Gary Seitz}



\thanks{The author was supported by the NSF grant DMS-2101791. }

\begin{abstract}
We explicitly compute the dimensions of certain idempotent truncations of RoCK blocks of cyclotomic quiver Hecke superalgebras. Equivalently, this amounts to a computation of the value of the Shapovalov form on certain explicit vectors in the basic  representations of twisted affine Kac-Moody Lie algebras of type $A$. 
\end{abstract}

\maketitle

\section{Introduction}

The goal of this paper is to obtain a technical (but neat) result computing the dimensions of certain idempotent truncations of RoCK blocks of cyclotomic quiver Hecke superalgebras. Equivalently, this amounts to a computation of the value of the Shapovalov form on certain explicit vectors in the basic  representation $V(\La_0)$ of twisted affine Kac-Moody Lie algebra  $\g$ of type $A_{2\ell}^{(2)}$. 

To state the result, let $\g$ be the Kac-Moody Lie algebra of type $A_{2\ell}^{(2)}$, with a normalized invariant form $(\cdot|\cdot)$ on the corresponding weight lattice $P$, and the Weyl group $W$. Let   
$$I=\{0,1,\dots,\ell\}, \quad J:=I\setminus\{\ell\},$$
$\{\al_i\mid i\in I\}$ be the simple roots, $\{\La_i\mid i\in I\}$ be the fundamental dominant weights, 
and $Q_+\subset P$ be the set of $\Z_{\geq 0}$-linear combinations of the simple roots. For the negative Chevalley generators $\{f_i\mid i\in I\}$ of $\g$, we have the divided powers $f_i^{(k)}:=f_i^k/k!\in U(\g)$ for $k\in \Z_{\geq 0}$. Fix a non-zero highest weight vector $v_+$ of the irreducible $\g$-module $V(\La_0)$ with highest weight $\La_0$. Let $(\cdot,\cdot)$ be the Shapovalov form on $V(\La_0)$  
such that $(v_+,v_+)=1$. 

For every $w\in W$ with reduced decomposition $w=r_{i_t}\cdots r_{i_1}$, setting 
\begin{equation}\label{EAk}
a_k:=( r_{i_{k-1}}\cdots r_{i_1}\La_0\mid \al_{i_k}^\vee)\qquad (k=1,\dots,t), 
\end{equation}
it is easy to see that $a_1,\dots,a_t\in\Z_{\geq 0}$ and $$v_w:=f_{i_t}^{(a_t)}\cdots f_{i_1}^{(a_1)}v_+$$ is a non-zero vector of the weight space $V(\La_0)_{w\La_0}$, 
which does not depend on the choice of reduced decomposition and satisfies 
$
(v_w,v_w)=1.
$

For every $j\in J$ and $m\in \Z_{\geq 0}$, we define the divided power monomial 
\begin{equation}\label{EFmj}
f(m,j):=f_j^{(m)}\cdots f_1^{(m)}f_0^{(2m)}f_1^{(m)}\cdots f_j^{(m)}f_{j+1}^{(2m)}\cdots f_{\ell-1}^{(2m)}f_\ell^{(m)}\in U(\g).
\end{equation}
(with $f(0,j)$ interpreted as $1$). Recalling the null-root $\de:=2\al_0+\dots+2\al_{\ell-1}+\al_\ell$, note that the monomial $f(m,j)$ has total weight $-m\de$. More generally, given $d, n\in \Z_{>0}$, a composition $\mu=(\mu_1,\dots,\mu_n)$ of $d$ into $n$ non-negative parts and a tuple $\bj=(j_1,\dots,j_n)\in J^n$ we define 
\begin{equation}\label{EFmubj}
f(\mu,\bj):=f(\mu_n,j_n)\cdots f(\mu_1,j_1)\in U(\g).
\end{equation}
For a composition $\om_d:=(1,\dots,1)$ of $d$, we also define 
\begin{equation}\label{Efom}
f(\om_d):=\sum_{\bj\in J^d}f(\om_d,\bj)=\sum_{j_1,\dots,j_d\in J}f(1,j_d)\cdots f(1,j_1). 
\end{equation}

If the weight space $V(\La_0)_{\La_0-\theta}$ for $\theta\in Q_+$ is non-zero, then we can write $\theta=\La_0-w\La_0+d\de$ for some $w$ in the Weyl group $W$ and a unique $d\in \Z_{\geq 0}$. 
We then say that $\theta$ is {\em RoCK}\, if $(\theta\mid \al_0^\vee)\geq 2d$ and $(\theta\mid \al_i^\vee)\geq d-1$ for $i=1,\dots,\ell$. 
This is equivalent to the cyclotomic quiver Hecke algebra $R_\theta^{\La_0}$ being a {\em RoCK block}, as defined in \cite[Section 4.1]{KlLi}. To each composition $\mu=(\mu_1,\dots,\mu_n)$ and tuple $\bj=(j_1,\dots,j_n)\in J^n$, we can define the corresponding divided power idempotents 
$e(\mu,\bj)\in R_\theta^{\La_0}$. The idempotents $e(\om_d,\bj)$ are then distinct and orthogonal to each other for distinct $\bj\in J^d$, so we also have the idempotent $e(\om_d):=\sum_{\bj\in J^d}e(\om_d,\bj)$.


\vspace{2mm}
\noindent
{\bf Main Theorem.}
{\em
Let weight $\theta=\La_0-w\La_0+d\de$ be RoCK, $n\in\Z_{>0}$, $\mu=(\mu_1,\dots,\mu_n)$ be a composition of $d$ with $n$ parts,  and  $\bj=(j_1,\dots,j_n)\in J^n$. 
Set  
\begin{equation}
\label{EmMuj}
|\mu,\bj|_{\ell-1}:=\sum_{\substack{1\leq r\leq n,\\ j_r=\ell-1}}\mu_r.
\end{equation} 
Then 
$$
\dim e(\mu,\bj) R_\theta^{\La_0}e(\om_d)=(f(\mu,\bj)v_w,f(\om_d)v_w)=\binom{d}{\mu_1\,\cdots\,\mu_n}\,4^{d-|\mu,\bj|_{\ell-1}}\,3^{|\mu,\bj|_{\ell-1}}. 
$$
}
\vspace{2mm}

The first equality in the theorem is a known consequence of the Kang-Kashiwara-Oh categorification, so the main work is to  prove of the second equality. This is proved in Theorem~\ref{TMainHalf'}.

The dimension formula for $e(\mu,\bj) R_\theta^{\La_0}e(\om_d)$ obtained in the Main Theorem is a shadow of the fact that the Gelfand-Graev idempotent truncation of a RoCK block $R_\theta^{\La_0}$ is isomorphic to a generalized Schur algebra corresponding to a certain Brauer tree algebra. In fact, our Main Theorem is a key step needed for the proof of this isomorphism in   \cite{KlSpinEKTwo}.

\subsection*{Acknowledgement} I am grateful to Matt Fayers for communicating the formula (\ref{EChi}) for the Fock space vectors corresponding to the irreducible supercharacters of double covers of symmetric groups.

\section{Shapovalov forms}

\subsection{Lie theoretic notation}
\label{SSLTN}
Let $\g$ be the Kac-Moody Lie algebra of type $A_{2\ell}^{(2)}$ (over $\C)$,  
see \cite[Chapter 4]{Kac}. 
We set 
$$p:=2\ell+1.$$

The Dynkin diagram of $\g$ has vertices labeled by $I=\{0,1,\dots,\ell\}$:
\vspace{5mm} 
$$
{\begin{picture}(330, 15)%
\put(6,5){\circle{4}}%
\put(101,2.45){$<$}%
\put(12,2.45){$<$}%
\put(236,2.42){$<$}%
\put(25, 5){\circle{4}}%
\put(44, 5){\circle{4}}%
\put(8, 4){\line(1, 0){15.5}}%
\put(8, 6){\line(1, 0){15.5}}%
\put(27, 5){\line(1, 0){15}}%
\put(46, 5){\line(1, 0){1}}%
\put(49, 5){\line(1, 0){1}}%
\put(52, 5){\line(1, 0){1}}%
\put(55, 5){\line(1, 0){1}}%
\put(58, 5){\line(1, 0){1}}%
\put(61, 5){\line(1, 0){1}}%
\put(64, 5){\line(1, 0){1}}%
\put(67, 5){\line(1, 0){1}}%
\put(70, 5){\line(1, 0){1}}%
\put(73, 5){\line(1, 0){1}}%
\put(76, 5){\circle{4}}%
\put(78, 5){\line(1, 0){15}}%
\put(95, 5){\circle{4}}%
\put(114,5){\circle{4}}%
\put(97, 4){\line(1, 0){15.5}}%
\put(97, 6){\line(1, 0){15.5}}%
\put(6, 11){\makebox(0, 0)[b]{$_{0}$}}%
\put(25, 11){\makebox(0, 0)[b]{$_{1}$}}%
\put(44, 11){\makebox(0, 0)[b]{$_{{2}}$}}%
\put(75, 11){\makebox(0, 0)[b]{$_{{\ell-2}}$}}%
\put(96, 11){\makebox(0, 0)[b]{$_{{\ell-1}}$}}%
\put(114, 11){\makebox(0, 0)[b]{$_{{\ell}}$}}%
\put(230,5){\circle{4}}%
\put(249,5){\circle{4}}%
\put(231.3,3.2){\line(1,0){16.6}}%
\put(232,4.4){\line(1,0){15.2}}%
\put(232,5.6){\line(1,0){15.2}}%
\put(231.3,6.8){\line(1,0){16.6}}%
\put(230, 11){\makebox(0, 0)[b]{$_{{0}}$}}%
\put(249, 11){\makebox(0, 0)[b]{$_{{1}}$}}%
\put(290,2){\makebox(0,0)[b]{\quad if $\ell = 1$.}}%
\put(175,2){\makebox(0,0)[b]{if $\ell \geq 2$, \qquad and}}%
\end{picture}}
$$

\vspace{3mm}
We have the standard Chevalley generators 
$\{e_i,f_i,h_i\mid i\in I\}$ of $\g$ and the Chevalley anti-involution 
$$
\si:\g\to\g, \ e_i\mapsto f_i,\  f_i\mapsto e_i,\ h_i\mapsto h_i.
$$

We have the {\em weight lattice} $P$ of $\g$, the subset of all {\em dominant integral weights} $P_+\subset P$, and the set $\{\alpha_i \mid i \in I\}\subset P$ of the {\em simple roots}\, of $\g$. 
We denote by $Q$ the sublattice of $P$ generated by the simple roots and set 
$$
Q_+:=\big\{\sum_{i\in I}m_i\al_i\mid m_i\in\Z_{\geq 0}\  \text{for all $i\in I$}\big\}\subset Q.
$$
For $\theta=\sum_{i\in I}m_i\al_i\in Q_+$, its {\em height}\, is  
$$
\height(\theta):=\sum_{i\in I}m_i.
$$
For $\theta\in Q_+$ of height $n$, we set
$$I^\theta:=\{(i_1,\cdots, i_n)\in I^n\mid \al_{i_1}+\dots+\al_{i_n}=\theta\}.$$
We have the null-root  
$$
\delta = \sum_{i=0}^{\ell - 1} 2 \alpha_i + \alpha_\ell
$$ 
with $\height(\de)=2\ell+1=p$. 

We denote by $(.|.)$ a normalized invariant form on $P$ whose Gram matrix with respect to the linearly independent set $\al_0,\al_1,\dots,\al_\ell$ is: 

\vspace{1mm}
$$
\left(
\begin{matrix}
2 & -2 & 0 & \cdots & 0 & 0 & 0 \\
-2 & 4 & -2 & \cdots & 0 & 0 & 0 \\
0 & -2 & 4 & \cdots & 0 & 0 & 0 \\
 & & & \ddots & & & \\
0 & 0 & 0 & \dots & 4 & -2& 0 \\
0 & 0 & 0 & \dots & -2 & 4& -4 \\
0 & 0 & 0 & \dots & 0 & -4& 8 \\
\end{matrix}
\right)
\quad
\text{if $\ell\geq 2$, and}
\quad
\left(
\begin{matrix}
2 & -4 \\
-4 & 8
\end{matrix}
\right)
\quad \text{if $\ell=1$.}
$$

\vspace{4mm}
Recall from \cite[\S\S3.7,3.13]{Kac} the (affine) {\em Weyl group} $W$ generated by the fundamental reflections $\{r_i\mid i\in I\}$  as a Coxeter group. The form $(.|.)$ is $W$-invariant, see \cite[Proposition 3.9]{Kac}.

For $\La\in P_+$, we denote by $V(\La)$ the {\em irreducible integrable $\g$-module}\, of highest weight $\La$. We fix a non-zero highest weight vector $v_+\in V(\La)_\La$. The {\em Shapovalov form}\, is the unique symmetric bilinear form $(\cdot,\cdot)$ on $V(\La)$ such that  $(v_+,v_+)=1$ and 
\begin{equation}\label{EF()}
(xv,w)=(v,\sigma(x)w)\qquad (x\in\g,\ v,w\in V(\La)).
\end{equation}

\begin{Lemma} \label{LWExt} 
Let $\La\in P_+$ and\, $w\in W$ with a reduced decomposition\, $w=r_{i_t}\cdots r_{i_1}$. Then:
\begin{enumerate}
\item[{\rm (i)}] $a_k:=(r_{i_{k-1}}\cdots r_{i_1}\La\mid \al_{i_k}^\vee)\in\Z_{\geq 0}
$ for all $k=1,\dots,t$;
\item[{\rm (ii)}] $v_w:=f_{i_t}^{(a_t)}\cdots f_{i_1}^{(a_1)}v_+$ is a non-zero vector of the weight space $V(\La)_{w\La}$, which does not depend on the choice of a reduced decomposition of $w$;
\item[{\rm (iii)}] $(v_w,v_w)=1$.
\end{enumerate}  
\end{Lemma}
\begin{proof}
This is well-known. For example, for all the claims, except the independence of a reduced decomposition, one can consult  \cite[Lemma 2.4.11]{KlLi}. One way to see the independence of a reduced decomposition is to first note that (iii) determines $v_w$ uniquely up to a sign, and if $f_{j_t}^{(b_t)}\cdots f_{j_1}^{(b_1)}v_+$ is such vector corresponding to another reduced decomposition, we cannot have $v_w=-f_{j_t}^{(b_t)}\cdots f_{j_1}^{(b_1)}v_+$, since, by the Kang-Kashiwara-Oh categorification \cite{KKO}, the vectors $f_{i_t}^{(a_t)}\cdots f_{i_1}^{(a_1)}v_+$ and $f_{j_t}^{(b_t)}\cdots f_{j_1}^{(b_1)}v_+$ correspond to modules over a certain algebra when $V(\La)$ is identified with a Grothendieck group.
\end{proof}

\subsection{Quantized enveloping algebra}

Let $q$ be an indeterminate and consider the ring $\Z[q,q^{-1}]$ of Laurent polynomials and the field $\C(q)$ of rational functions.  
For $i\in I$ and $n\in\Z_{\geq 0}$, we have the following elements of $\Z[q,q^{-1}]$: 
\begin{equation}\label{EQInt}
q_i:=q^{(\al_i\mid \al_i)/2},\quad [n]_i:=\frac{q_i^{n}-q_i^{-n}}{q_i-q_i^{-1}},\quad [n]_i^!:=[1]_i[2]_i\cdots[n]_i. 
\end{equation}

Let $U_q(\g)$ be the {\em quantized enveloping algebra of type $A_{2\ell}^{(2)}$}, i.e the the associative unital $\C(q)$-algebra with generators $\{E_i,F_i,K_i^{\pm 1}\mid i\in I\}$ subject only to the quantum Serre relations:
\begin{eqnarray*}
T_iE_jT_i^{-1}&=&q^{(\al_i\mid \al_j)}E_j,
\\
T_iF_jT_i^{-1}&=&q^{-(\al_i\mid \al_j)}F_j, 
\\
E_iF_j-F_jE_i&=&\de_{i,j}\frac{T_i-T_i^{-1}}{q_i-q_i^{-1}},
\\
\sum_{r=0}^{1-a_{i,j}}(-1)^rE_i^{(r)}E_jE_i^{(1-a_{i,j}-r)}&=&0\qquad(i\neq j),
\\
\sum_{r=0}^{1-a_{i,j}}(-1)^rF_i^{(r)}F_jF_i^{(1-a_{i,j}-r)}&=&0\qquad(i\neq j),
\end{eqnarray*}
where we have set
$
T_i:=K_i^{(\al_i\mid \al_i)/2},\ E_i^{(r)}:=\frac{E_i}{[r]_i^!},\  F_i^{(r)}:=\frac{F_i}{[r]_i^!}.
$

\vspace{1.5mm}
There is a $\C(q)$-linear anti-involution $\sigma_q:U_q(\g)\to U_q(\g)$ with 
\begin{align*}
\sigma_q\,:\, E_i&\mapsto q_iF_iT_i^{-1}=q_i^{-1}T_i^{-1}F_i,\quad 
F_i\mapsto q_i^{-1}T_iE_i=q_iE_iT_i,
\quad 
T_i\mapsto T_i.
\end{align*}

For $\La\in P_+$, we denote by $V_q(\La)$ the {\em irreducible integrable module}\, for $U_q(\g)$ of highest weight $\La$. We fix a non-zero highest weight vector $v_{+,q}\in V_q(\La)_\La$, so  $E_i v_{+,q}=0$ and $T_iv_{+,q}=q^{(\al_i\mid \La)}v_{+,q}$ for all $i\in I$. 
There is a unique symmetric bilinear form $(\cdot,\cdot)_q$ on $V_q(\La)$ such that  $(v_{+,q},v_{+,q})_q=1$ and 
\begin{equation}\label{EF()q}
(xv,w)_q=(v,\sigma_q(x)w)_q\qquad(x\in U_q(\g),\ v,w\in V_q(\La)),
\end{equation} 
cf. \cite[Appendix D]{KMPY}. We refer to $(\cdot,\cdot)_q$ as the (quantum) Shapovalov form.


\subsection{Putting $q$ to $1$}
\label{SSQ1}
Let $V_q$ be a $U_q(\g)$-module which decomposes as a direct sum of finite dimensional integral weight spaces $V_{q}=\bigoplus_{\mu\in P}V_{q,\mu}$. Suppose also that there exists a full rank $\Z[q,q^{-1}]$-sublattice $V_{q,\mu,\Z[q,q^{-1}]}\subset V_{q,\mu}$ for every $\mu$ such that $V_{q,\Z[q,q^{-1}]}:=\bigoplus_{\mu\in P} V_{q,\mu,\Z[q,q^{-1}]}$ is stable under all $E_i$ and $F_i$. Considering $\C$ as a $\Z[q,q^{-1}]$-module with $q$ acting as $1$, we change scalars to get the complex vector space $V_q\mid_{q=1}:=\C\otimes_{\Z[q,q^{-1}]}V_{q,\Z[q,q^{-1}]}$ with linear operators $e_i:=1\otimes E_i$, $f_i:=1\otimes F_i$ and 
$h_i:=1\otimes \frac{T_i-T_i^{-1}}{q_i-q_i^{-1}}$ for all $i$. These linear operators are easily checked to satisfy the Serre relations for $\g$, cf. \cite[5.13]{JantzenB}, \cite[Theorem 4.12 and \S4.14]{Lusztig}. Thus $V_q\mid_{q=1}$ becomes a $\g$-module. (This module  depends on the choice of the $\Z[q,q^{-1}]$-sublattice).

Moreover, given a symmetric bilinear form $(\cdot,\cdot)_q$ on $V_q$ which satisfies (\ref{EF()q}) and is 
$\Z[q,q^{-1}]$-valued on $V_{q,\Z[q,q^{-1}]}$, we obtain, extending scalars, a symmetric bilinear form $(\cdot,\cdot)_q\mid_{q=1}$ on $V_q\mid_{q=1}$ satisfying (\ref{EF()}). 

The above constructions can be applied to the irreducible modules $V_q(\La)$ with highest weight $\La\in P_+$ by considering the $\Z[q,q^{-1}]$-sublattice spanned by all vectors of the form $F_{i_1}\cdots F_{i_r}v_+$. Taking into account that the formal characters of $V_q(\La)$ and $V(\La)$ agree by \cite[Theorem 4.12 and \S4.14]{Lusztig}, 
this shows that there is a unique isomorphism $V(\La)\iso V_q(\La)\mid_{q=1}$ mapping $v_+$ onto $1\otimes v_{+,q}$. Moreover, identifying $V(\La)$ and $V_q(\La)\mid_{q=1}$ under this isomorphism, the quantum Shapovalov form $(\cdot,\cdot)_q$ yields the usual Shapovalov form $(\cdot,\cdot)$, i.e. 
$(\cdot,\cdot)_q\mid_{q=1}=(\cdot,\cdot)$. In particular:

\begin{Lemma} \label{LShapqShap}
Let $(\cdot,\cdot)_q$ be the Shapovalov form on $V_q(\La)$ and $(\cdot,\cdot)$ be the Shapovalov form on $V(\La)$. Then for all $i_1,\dots,i_n,j_1,\dots,j_n\in I$, we have $(F_{i_1}\cdots F_{i_n}v_{+,q},F_{j_1}\cdots F_{j_n}v_{+,q})_q$ belongs to $\Z[q,q^{-1}]$ and 
$$
(f_{i_1}\cdots f_{i_n}v_+,f_{j_1}\cdots f_{j_n}v_+)=(F_{i_1}\cdots F_{i_n}v_{+,q},F_{j_1}\cdots F_{j_n}v_{+,q})_q\mid_{q=1}.
$$
\end{Lemma}

Another example of passing from $V_q$ to $V_q\mid_{q=1}$ will be considered in Section~\ref{SFock}.

\section{Combinatorics}

Recall that  we have set $p=2\ell+1$. 

\subsection{Partitions, multipartitions, tableaux}
\label{SSComb}
We denote by $\Par$ the set of all partitions and by $\Par(n)$ the set of all partitions of $n\in\Z_{\geq 0}$. For $\la\in\Par(n)$ we write $|\la|=n$. 
Collecting equal parts of $\la\in\Par$, we  can write it in the form 
\begin{equation}\label{EStrictForm}
\la=(l_1^{m_1},\dots, l_k^{m_k})\ \text{with $l_1>\dots>l_k>0$ and $m_1,\dots,m_k\geq 1$}.
\end{equation} 
We then denote 
\begin{equation}\label{ELaNorm}
\norm{\la}_q\,:=\prod_{r\,\, \text{with\,\,$p| l_r$}}\,\prod_{s=1}^{m_r}(1-(-q^2)^{s}),
\end{equation}
and 
\begin{equation}\label{Ehp}
h(\la)\,:=\sum_{r=1}^k m_r,\quad 
h_p(\la)\,:=\sum_{r\,\, \text{with\,\,$p| l_r$}}\,{m_r}.
\end{equation}
In other words, 
$h(\la)$ is the number of (positive) parts of $\la$ and 
$h_p(\la)$ is the number of (positive) parts of $\la$ divisible by $p$. 
If $m_r>1$ implies $p\mid l_r$ for all $1\leq r\leq k$ then $\la$ is called {\em $p$-strict}. Note that $0$-strict also makes sense and means simply {\em strict}, i.e. all parts are distinct. 
We denote by $\Par_p(n)$ the set of all $p$-strict partitions of $n$,  and let 
$
\Par_p:=\bigsqcup_{n\geq 0}\Par_p(n).
$
We use the similar notation $\Par_0(n)$ and $\Par_0$ for strict partitions. 

For $\la\in\Par_0$ we define its {\em parity}:
\begin{equation}\label{EAla}
\sfp_\la:=
\left\{
\begin{array}{ll}
1 &\hbox{if $\la$ has odd number of positive even parts,}\\
0 &\hbox{otherwise.}
\end{array}
\right.
\end{equation}

Let $\la$ be a $p$-strict partition. As usual, we identify $\la$ with its {\em Young diagram}\index{Young diagram} 
$$\la=\{(r,s)\in\Z_{>0}\times \Z_{>0}\mid s\leq \la_r\}.$$ 
We refer to the element $(r,s)\in\Z_{>0}\times \Z_{>0}$ as the {\em node}\, in row $r$ and column $s$. 
We define a preorder `$\leq$' on the nodes via 
$
(r,s)\leq(r',s')
$
if and only if $s\leq s'$. 
For partitions (equivalently Young diagrams) $\al\subseteq \la$, we define 
\begin{equation}\label{EqAlLa}
q(\la/\al):=|\{r\in\Z_{>0}\mid \la\setminus \al\ \text{has a node in column $r$ but not in column $r+1$}\}|.
\end{equation}

We label the nodes with the elements of the set $I$ as follows: the labeling follows the repeating pattern
$
0,1,\dots,\ell-1,\ell,\ell-1,\dots,1,0,
$
starting from the first column and going to the right, 
see Example~\ref{Ex140821} below. If a node $\ttA\in \la$ is labeled with $i$, we say that $\ttA$ has {\em residue $i$} and write $\Res\, \ttA=i$. 
Recalling $\al_i$'s and $Q_+$ from \S\ref{SSLTN}, define the {\em residue content}\, of $\la$ to be 
$$
\cont(\la):=\sum_{\ttA\in \la}\al_{\Res\, \ttA}\in Q_+.
$$ 
We always write $\cont(\la)=\sum_{i\in I}c^\la_i\al_i$, and 
\begin{equation}\label{Emneq0}
c^\la_{\neq 0}:=c^\la_1+\dots+c^\la_\ell=|\la|-c^\la_0.
\end{equation}

Following \cite{Morris,LT}, we can associate to every $\la\in\Par_p$ its {\em $\bar p$-core} 
$$\core(\la)\in\Par_p$$ 
obtained from $\la$ by removing certain nodes.  It is clear from the definition that the number of nodes removed to go from $\la$ to $\core(\la)$ is divisible by $p$, so the {\em $\bar p$-weight}\, of~$\la$ 
$$
\wt(\la):=(|\la|-|\core(\la)|)/p
$$ 
is a non-negative integer. 
A partition $\rho\in\Par_p$ is called a {\em $\bar p$-core}\, if $\core(\rho)=\rho$. By \cite[Lemma 3.1.39]{KlLi}, we have:

\begin{Lemma} \label{LCorew} 
A $p$-strict partition $\la$ is a $\bar p$-core if and only if $\cont(\la)=\La_0-w\La_0$ for some $w\in W$. 
\end{Lemma}

\vspace{2mm}
\begin{Example}\label{Ex140821}
Let $\ell=2$, so $p=5$. The partition $\la=(16, 11,10,10,9,4,1)$ is $5$-strict. The residues of the nodes are as follows:
\vspace{5mm}
$$
\begin{ytableau}
$0$ & $1$ & $2$ & $1$ & $0$ & $0$& $1$ & $2$ & $1$ & $0$ & $0$ & $1$ & $2$ & $1$ & $0$ & $0$\cr 
$0$ & $1$ & $2$ & $1$ & $0$ & $0$ & $1$ & $2$ & $1$ & $0$ & $0$ \cr
$0$ & $1$ & $2$ & $1$ & $0$ & $0$ & $1$ & $2$ & $1$ & $0$  \cr
$0$ & $1$ & $2$ & $1$ & $0$ & $0$ & $1$ & $2$ & $1$ & $0$  \cr
$0$ & $1$ & $2$ & $1$ & $0$ & $0$ & $1$ & $2$ & $1$ \cr
$0$ & $1$ & $2$ & $1$ \cr
$0$ \cr
\end{ytableau}\vspace{4 mm}
$$
The $\bar 5$-core of $\la$ is $(1)$.
\end{Example}

The partition $\la\in\Par_p$ is determined by its $\bar p$-core $\core(\la)$ and $\bar p$-quotient 
\begin{equation}\label{EPQuot}
\quot(\la)=(\la^{(0)},\dots,\la^{(\ell)})
\end{equation}
 which is an {\em $I$-multipartition}\, of $d$, i.e. $\la^{(0)},\dots,\la^{(\ell)}$ are partitions and $|\la^{(0)}|+|\la^{(1)}|+\dots+|\la^{(\ell)}|=d$. We denote the set of all such multipartitions by $\Par^I(d)$, and set
$$
\Par^I:=\bigsqcup_{d\geq 0}\Par^I(d).
$$
We refer the reader to \cite[p.27]{MorYas} and \cite[\S2.3b]{KlLi} for details on this. 
For a $\bar p$-core partition $\rho$, we denote 
$$
\Par_p(\rho,d):=\{\la\in\Par_p\mid \core(\la)=\rho\ \text{and}\ \wt(\la)=d\}.
$$ 
The map
\begin{equation}\label{EBijection0}
\Par_p(\rho,d)\to \Par^I(d),\ \la\mapsto\cont(\la) 
\end{equation}
is a bijection, see \cite[Theorem 2]{MorYas}. 
The condition $\la\in \Par_p(\rho,d)$ is equivalent to $\cont(\la)=\cont(\rho)+d\de$, cf. \cite[Lemma 2.3.9]{KlLi}.

A multipartition $\ula=(\la^{(0)},\la^{(1)},\dots,\la^{(\ell)})\in\Par^I$ will be called {\em strict}\, if its $0$th component $\la^{(0)}$ is a strict partition (and $\la^{(1)},\dots,\la^{(\ell)}$ are arbitrary partitions). For $d\in\Z_{\geq 0}$, we denote by $\Par_0^I(d)$ the set of all strict multipartitions of $d$. Note that $\quot(\la)\in \Par_0^I(d)$ if and only if $\la\in \Par_0(\rho,d)$, so the bijection  (\ref{EBijection0}) restricts to the bijection
\begin{equation}\label{EBijection}
\Par_0(\rho,d)\to \Par^I_0(d),\ \la\mapsto\cont(\la).
\end{equation}

We identify a multipartition $\ula$ with its Young diagram 
$$\ula=\{(i,r,s)\in I\times \Z_{>0}\times \Z_{>0}\mid s\leq \la^{(i)}_r\}.$$ 
We refer to the element $(i,r,s)\in I\times \Z_{>0}\times \Z_{>0}$ as the {\em node}\, in row $r$ and column $s$ of component $i$. For each $i$, we consider (the Young diagram of) the partition $\la^{(i)}$ as a subset $\la^{(i)}\subseteq \ula$ consisting of the notes of $\ula$ in its $i$th component.

Let $n\in\Z_{>0}$ and $d\in\Z_{\geq 0}$. A {\em composition}\, of $d$ with $n$ parts is a tuple $\mu=(\mu_1,\dots,\mu_n)$ with $\mu_1,\dots,\mu_n\in\Z_{\geq 0}$ such that $\mu_1+\dots+\mu_n=d$. We denote by $\La(n,d)$ the set of all compositions of $d$ with $n$ parts. 
We will need a special composition of $d$ with $d$ parts:
\begin{equation}\label{Eomd}
\om_d:=(1^d)=(1,\dots,1).
\end{equation}
Recall that $J$ denotes $I\setminus\{\ell\}$. 
A {\em colored composition of $d$ with $n$ parts}\, is a pair $(\mu,\bj)$ where $\mu=(\mu_1,\dots,\mu_n)$ is a composition of $d$ with $n$ parts and $\bj=(j_1,\dots,j_n)\in J^n$. We denote by $\La^\col(n,d)$ the set of all colored compositions of $d$ with $n$ parts. 

Let $\ula=(\la^{(0)},\la^{(1)},\dots,\la^{(\ell)})\in\Par^I(d)$ and $(\mu,\bj)\in \La^\col(n,d)$. A {\em colored tableau of shape $\ula$ and type $(\mu,\bj)$} is a function 
$\T:\ula\to \Z_{>0}$ such that

\vspace{1mm}
(1) $\T(i,r,s)\leq \T(i,r,s+1)$ and $\T(i,r,s)\leq \T(i,r+1,s)$ whenever these make sense;

(2) for all $k=1,\dots,n$, we have $|\T^{-1}(\{k\})|=\mu_k$ and $\T^{-1}(\{k\})\subseteq \la^{(j_k)}\sqcup \la^{(j_k+1)}$;

(3) for all $k=1,\dots,n$, no two nodes of $\T^{-1}(\{k\})\cap\la^{(j_k)}$ are in the same column, and no two nodes of $\T^{-1}(\{k\})\cap\la^{(j_k+1)}$ are in the same row.

\vspace{1mm}
We denote by $\CT(\ula;\mu,\bj)$ the set of all colored tableaux of shape $\ula$ and type $(\mu,\bj)$. 
For $\T\in \CT(\ula;\mu,\bj)$ and $1\leq k\leq n$ we denote by $q_k(\T)$ the number of positive integers $r$ such that 
$\T^{-1}(k)\cap\la^{(0)}$ contains a node in column $r$ but not in column $r+1$. We then set 
$
q(\T)=q_1(\T)+\dots+q_n(\T) 
$ 
and denote
\begin{equation}\label{EK}
K(\ula;\mu,\bj):=\sum_{\T\in \CT(\ula;\mu,\bj)}2^{q(\T)}.
\end{equation}

\subsection{Addable and removable nodes}
Let $\la$ be a $p$-strict partition and $i \in I$.
A node $\ttA
\in \la$ is called {\em $i$-removable}\, (for $\la$) if one of the following
holds:
\begin{enumerate}
\item[(R1)] $\Res\, \ttA = i$ and
$\la_\ttA:=\la\setminus \{\ttA\}$\index{$\la_\ttA$} is again a $p$-strict partition; such $\ttA$'s are also called {\em properly $i$-removable};
\index{properly removable node}
\item[(R2)] the node $\ttB$ immediately to the right of $\ttA$
belongs to $\la$,
$\Res\, \ttA = \Res\, \ttB = i$,
and both $\la_\ttB = \la \setminus  \{\ttB\}$ and
$\la_{\ttA,\ttB} := \la \setminus  \{\ttA,\ttB\}$ are $p$-strict partitions.
\end{enumerate}
A node $\ttB\notin\la$ is called 
{\em $i$-addable}\, (for $\la$) if one of the following holds:
\begin{enumerate}
\item[(A1)] $\Res\, \ttB = i$ and
$\la^\ttB:=\la\cup\{\ttB\}$\index{$\la^\ttB$} is again a $p$-strict partition; such $\ttB$'s are also called {\em properly $i$-addable};\index{properly addable node}
\item[(A2)] 
the node $\ttA$
immediately to the left of $\ttB$ does not belong to $\la$,
$\Res\, \ttA = \Res\, \ttB = i$, and both 
$\la^\ttA = \la \cup \{\ttA\}$ and 
$\la^{\ttA, \ttB} := \la \cup\{\ttA,\ttB\}$ are $p$-strict partitions.
\end{enumerate}
We note that (R2) and (A2) above are only possible if $i = 0$.
For $i \in I$, we denote by $\Add_i(\la)$\index{a@$\Add_i(\la)$} (resp. $\Rem_i(\la)$)\index{r@$\Rem_i(\la)$} the set of all $i$-removable (resp. $i$-addable) nodes for $\la$. 
We also denote by $\PA_i(\la)$\index{p@$\PA_i(\la)$} (resp. $\PR_i(\la)$)\index{p@$\PA_i(\la)$}\index{p@$\PR_i(\la)$} the set of all properly $i$-removable (resp. properly $i$-addable) nodes for $\la$.

Let $\la\in\Par_p$ be written in the form (\ref{EStrictForm}). Suppose $\ttA\in \PR_i(\la)$. Then there is $1\leq r\leq k$ such that $\ttA=(m_1+\dots+m_r,l_r)$.  
Recalling the preorder `$\leq$' on the nodes defined above, we set  
\begin{align*}
\eta_\ttA(\la)&:=\sharp\{\ttC\in\Rem_i(\la)\mid \text{$\ttC>\ttA$}\}
-\sharp\{\ttC\in\Add_i(\la)\mid \text{$\ttC>\ttA$}\},
\index{e@$\eta_\ttA(\la)$}
\\
\zeta_\ttA(\la)&:=
\left\{
\begin{array}{ll}
(1-(-q^2)^{m_r}) &\hbox{if $p\mid l_r$,}\\
1 &\hbox{otherwise.}
\end{array}
\right.
\index{z@$\zeta_\ttA(\la)$}
\\
d_\ttA(\la)&:= q_i^{\eta_\ttA(\la)} \zeta_\ttA(\la).
\index{d@$d_\ttA(\la)$}
\end{align*}
Note that
\begin{equation}\label{EDQE1}
d_\ttA(\la)\mid_{q=1}=
\left\{
\begin{array}{ll}
0 &\hbox{if $p\mid l_r$ and $m_r$ is even,}\\
2 &\hbox{if $p\mid l_r$ and $m_r$ is odd,}\\
1  &\hbox{otherwise.}
\end{array}
\right.
\end{equation}

Suppose $\ttB\in\PA_i(\la)$. Then there is $r$ such that $1\leq r\leq k+1$ and $\ttB=(m_1+\dots+m_{r-1}+1,l_r+1)$, where we interpret $l_{k+1}$ as $0$.   We define 
\begin{align}
\label{EEtaDef}
\eta^\ttB(\la)&:=\sharp\{\ttC\in\Add_i(\la)\mid \text{$\ttC<\ttB$}\}
-\sharp\{\ttC\in\Rem_i(\la)\mid \text{$\ttC<\ttB$}\},
\index{e@$\eta^\ttB(\la)$}
\\
\zeta^\ttB(\la)&:=
\left\{
\begin{array}{ll}
(1-(-q^2)^{m_r}) &\hbox{if $r\leq k$ and $p\mid l_r$,}\\
1 &\hbox{otherwise.}
\end{array}
\right.
\index{z@$\zeta^\ttB(\la)$}
\\
d^\ttB(\la)&:= q_i^{\eta^\ttB(\la)} \zeta^\ttB(\la).
\index{d@$d^\ttB(\la)$}
\end{align}

Note that
\begin{equation}\label{EDQF1}
d^\ttB(\la)\mid_{q=1}=
\left\{
\begin{array}{ll}
0 &\hbox{if $r\leq k$, $p\mid l_r$ and $m_r$ is even,}\\
2 &\hbox{if $r\leq k$, $p\mid l_r$ and $m_r$ is odd,}\\
1  &\hbox{otherwise.}
\end{array}
\right.
\end{equation}

\begin{Example} \label{Ex1} 
{\rm 
Let $\ell=2$ so $p=5$.
The partition $\la=(5, 5,2)$ is $5$-strict,
and the residues of its boxes are labeled on the diagram below: 

\vspace{2mm} 
$$
\begin{ytableau}
$0$ & $1$ & $2$ & $1$ & $0$ \cr 
$0$ & $1$ & $2$ & $1$ & $0$\cr
$0$ & $1$\cr
\end{ytableau}\vspace{2 mm}
$$

\vspace{3 mm}
\noindent
The only $0$-removable node is marked as $\ttA_1$, and the $0$-addable nodes are marked as $\ttB_1,\ttB_2$:

\vspace{2mm} 
$$
\ytableausetup{mathmode}
\begin{ytableau}
\, &  &  &  &  & \none[\ttB_2] \cr 
 &  &  &  & \ttA_1 \cr
 & \cr
 \none[\ttB_1]
\end{ytableau}\vspace{2 mm}
$$

\vspace{1mm} 
\noindent
We have $d^{\ttB_1}(\la)=1$ and $d^{\ttB_2}(\la)=(1-q^4)$. On the other hand for the partition $\mu=(5)$ and the node $\ttB=(1,6)$, we have $d^\ttB(\mu)=(1+q^2)$. 
}
\end{Example}

\subsection{Symmetric functions}
We denote by $\La$ the algebra of symmetric functions in the variables $x_1,x_2,\dots$ over $\C$ with  the basis $$\{\pzs_\la\mid \la\in\Par\}$$ of {\em Schur functions}\, and the inner product $\lan\cdot,\cdot\ran$ such that $\lan \pzs_\la,\pzs_\mu\ran=\de_{\la,\mu}$, see \cite{Mac}. We also have the 
{\em monomial symmetric functions} $\pzm_\la$, the 
{\em elementary symmetric functions} $\pze_r=\pzs_{(1^r)}$ and the {\em complete symmetric functions} $\pzh_r=\pzs_{(r)}$ for $r\in \Z_{\geq 0}$. {\em Pieri's Rules}\, \cite[(5.16),(5.17)]{Mac} say that
\begin{equation}\label{EPieri}
\pzs_\la \pzh_r=\sum_\mu \pzs_\mu\quad\text{and}\quad
\pzs_\la \pze_r=\sum_\nu \pzs_\nu
\end{equation}
where the first sum is over all partitions $\mu$ obtained by adding $r$ nodes to $\la$ with no two nodes added in the same column, and 
the second sum is over all partitions $\nu$ obtained by adding $r$ nodes to $\la$ with no two nodes added in the same row.

Suppose that for $s_1,\dots,s_t\in\Z_{\geq 0}$, we have that $\pzf_{s_u}=\pze_{s_u}$ or $\pzh_{s_u}$. 
Under the characteristic map \cite[I.7]{Mac}, the symmetric function $\pzf_{s_1}\cdots \pzf_{s_t}$ corresponds to an induced representation of the symmetric group $\Si_{s_1+\dots+s_t}$ of dimension given by the multinomial coefficient $\binom{s_1+\dots+s_t}{s_1\,\,\cdots\,\, s_t}$, while the symmetric function $\pzs_{(1)}^{\hspace{.2mm}r}$ with $r\in \Z_{\geq 0}$ corresponds to the regular representation of the symmetric group $\Si_r$. Hence we have 
\begin{equation}\label{EPsiReg}
(\pzf_{s_1}\cdots \pzf_{s_t},\pzs_{(1)}^{\hspace{.2mm}r})=
\left\{
\begin{array}{ll}
\binom{r}{s_1\,\cdots\, s_t} &\hbox{if $s_1+\dots+s_t=r$,}\\
0 &\hbox{otherwise.}
\end{array}
\right.
\end{equation}

Denoting by $\pzp_r\in\La$ the $r$th power sum symmetric function, let $\Om$ be the (unital) subalgebra of $\La$ generated by $\pzp_1,\pzp_3,\pzp_5,\dots$. Then $\Om$ has bases 
$$
\{\pzP_\la\mid \la\in\Par_0\}\quad \text{and}\quad \{\pzQ_\la\mid \la\in\Par_0\},
$$ 
where $\pzP_\la$ and $\pzQ_\la$ are {\em Schur's $P$-} and {\em $Q$-symmetric functions}, see \cite[\S6]{St}. We have $\pzP_\la=2^{-h(\la)}\pzQ_\la$ for all $\la\in\Par_0$. Let $[\cdot,\cdot]$ be inner product on $\Om$ such that $[\pzP_\la,\pzQ_\la]=\de_{\la,\mu}$ for all $\la,\mu\in\Par_0$, see \cite[\S\S5,6]{St}.

We also have the symmetric functions 
$$\pzq_r=2\pzP_{(r)}\in\Om\qquad(r\in \Z_{> 0})$$ (and $\pzq_0:=1$), see \cite[(5.3),(6.6)]{St}.
We have the analogue of the Pieri's Rule (which goes back to \cite{MorrisProd} but can be most easily seen from \cite[Theorem 8.3]{St}):
\begin{equation}\label{EPieriq}
\pzP_\la \pzq_r=\sum_\mu 2^{q(\mu/\la)}\pzP_\mu,
\end{equation}
where the sum is over all strict partitions $\mu$ obtained by adding $r$ nodes to $\la$ with no two nodes added in the same column and $q(\mu/\la)$ is as in (\ref{EqAlLa}). 

We note that by \cite[Proposition 5.6(b)]{St}, the inner product $[\pzq_{s_1}\dots \pzq_{s_t},\pzq_1^r]$ is the coefficient of $\pzm_{(s_1,\dots,s_r)}$ in $\pzq_1^r=(2x_1+2x_2+\dots)^r$ (we may assume that $s_1\geq\dots\geq s_r$ so $(s_1,\dots,s_r)$ is a partition), whence 
\begin{equation}\label{EQQ}
[\pzq_{s_1}\dots \pzq_{s_t},\pzq_1^r]=
\left\{
\begin{array}{ll}
2^{r}\binom{r}{s_1\,\cdots\, s_t} &\hbox{if $s_1+\dots+s_t=r$,}\\
0 &\hbox{otherwise.}
\end{array}
\right.
\end{equation}

We consider the algebra 
$$\Sym^I:=\Om\otimes \La^{(1)}\otimes\dots\otimes \La^{(\ell)}
$$
where each algebra $\La^{(i)}$ is just a copy of $\La$. This has bases
\begin{equation}\label{EPi}
\{\pi_\ula:=\pzP_{\la^{(0)}}\otimes \pzs_{\la^{(1)}}\otimes\dots\otimes \pzs_{\la^{(\ell)}}\mid \ula=(\la^{(0)},\dots,\la^{(\ell)})\in\Par^I_0\}
\end{equation}
and 
\begin{equation}\label{EKa}
\{\ka_\ula:=\pzQ_{\la^{(0)}}\otimes \pzs_{\la^{(1)}}\otimes\dots\otimes \pzs_{\la^{(\ell)}}\mid \ula=(\la^{(0)},\dots,\la^{(\ell)})\in\Par^I_0\}
\end{equation}
which are dual with respect to the inner product $(\cdot,\cdot)_{\Sym}$ defined as 
\begin{equation}\label{ESymInnerProd}
(f_0\otimes f_1\otimes\dots\otimes f_\ell,g_0\otimes g_1\otimes\dots\otimes g_\ell)_{\Sym}:=[f_0,g_0]\lan f_1,g_1\ran\cdots\lan f_\ell,g_\ell\ran.
\end{equation}

Let $(\mu,\bj)\in \La^\col(n,d)$. Recalling (\ref{EK}), set  
\begin{equation}\label{EPiMuBj}
\Pi_{\mu,\bj}:=\sum_{\ula\in\Par_0^I(d)}K(\ula;\mu,\bj)\pi_\ula.
\end{equation}

\vspace{1mm}
\begin{Example}
Let $n=1$ so $(\mu,\bj)$ is of the form $((d),j)\in \La^\col(1,d)$. Then 
\begin{equation}\label{EPiMuBjn=1}
\Pi_{(d),j}:=
\left\{
\begin{array}{ll}
1\otimes \pze_{d}\otimes 1^{\otimes \ell-1}+
2\sum_{k=1}^d \pzP_{(k)}\otimes \pze_{d-k}\otimes 1^{\otimes \ell-1} &\hbox{if $j=0$,}\\
\sum_{k=0}^d 1^{\otimes j}\otimes \pzh_{k}\otimes \pze_{d-k}\otimes 1^{\otimes \ell-1-j} &\hbox{if $1\leq j<\ell$.}
\end{array}
\right.
\end{equation}
\end{Example}

\vspace{2mm}
\begin{Lemma} \label{LPiSplit}
Let $(\mu,\bj)\in \La^\col(n,d)$. Suppose $n\geq 2$ and set 
$\nu:=(\mu_1,\dots,\mu_{n-1})$, $\bk:=(j_1,\dots,j_{n-1})$, $m:=\mu_n$, $j:=j_n$. Then 
$$\Pi(\mu,\bj)=\Pi(\nu,\bk)\Pi((m),j).$$ 
\end{Lemma}
\begin{proof}
Suppose $j\neq 0$. 
Using (\ref{EPiMuBjn=1}), we see that $\Pi(\nu,\bk)\Pi((m),j)$ equals
\begin{align*}
&\Big(\sum_{\ual\in\Par_0^I(d-m)}K(\ual;\nu,\bk)\pi_\ual\,\Big)\Big(\sum_{k=0}^{m} 1^{\otimes j}\otimes \pzh_{k}\otimes \pze_{m-k}\otimes 1^{\otimes \ell-1-j}\,\Big)
\\
=\,&\sum_{\substack{\ual\in\Par_0^I(d-m)\\ 0\leq k\leq m}}K(\ual;\nu,\bk)\pzP_{\al^{(0)}}\otimes \pzs_{\al^{(1)}}\otimes\dots \otimes \pzs_{\al^{(j)}}\pzh_k\otimes \pzs_{\al^{(j+1)}}\pze_{m-k}\otimes\dots \otimes \pzs_{\al^{(\ell)}}
\\
=\,&\sum_{\ula\in\Par_0^I(d)}K(\ula;\mu,\bj)\pzP_{\la^{(0)}}\otimes \pzs_{\la^{(1)}}\otimes\dots \otimes \pzs_{\la^{(j)}}\otimes \pzs_{\la^{(j+1)}}\otimes\dots \otimes \pzs_{\la^{(\ell)}},
\end{align*}
where the last equality follows from Pieri's Rules (\ref{EPieri}). 

Suppose now that $j= 0$. 
Using (\ref{EPiMuBjn=1}), we see that $\Pi(\nu,\bk)\Pi((m),0)$ equals
\begin{align*}
&\Big(\sum_{\ual\in\Par_0^I(d-m)}K(\ual;\nu,\bk)\pi_\ual\,\Big)\Big(\sum_{k=0}^d \pzq_k\otimes \pze_{d-k}\otimes 1^{\otimes \ell-1}\,\Big)
\\
=\,&\sum_{\substack{\ual\in\Par_0^I(d-m)\\ 0\leq k\leq m}}K(\ual;\nu,\bk)\pzP_{\al^{(0)}}\pzq_k\otimes \pzs_{\al^{(1)}}\pze_{d-k}\otimes \pzs_{\al^{(2)}}\otimes \dots \otimes \pzs_{\al^{(\ell)}}
\\
=\,&\sum_{\ula\in\Par_0^I(d)}K(\ula;\mu,\bj)\pzP_{\la^{(0)}}\otimes \pzs_{\la^{(1)}}\otimes\dots \otimes \pzs_{\la^{(\ell)}},
\end{align*}
where the last equality follows from Pieri's Rules (\ref{EPieriq}) and  (\ref{EPieri}).
\end{proof}

Recalling (\ref{Eomd}), we let 
\begin{equation}\label{}
\Pi_{\om_d}:=\sum_{\bj\in J^d}\Pi_{\om_d,\bj}.
\end{equation}

\begin{Corollary} \label{C1^d}
We have
$$
\Pi_{\om_d}=\sum_{k_0+k_1+\dots+k_\ell=d}2^{k_1+\dots+k_{\ell-1}}{{d}\choose{k_0\,k_1\,\cdots\,k_\ell}}\pzq_1^{\hspace{.2mm}k_0}\otimes \pzs_{(1)}^{\hspace{.2mm}k_1}\otimes\dots\otimes \pzs_{(1)}^{\hspace{.2mm}k_\ell}.
$$
\end{Corollary}
\begin{proof}
We apply induction on $d$. In the base case $d=1$,
by (\ref{EPiMuBjn=1}) and using $\pzq_1=2\pzP_{(1)}$ we have:
\begin{align*}
\Pi_{\om_1}=\,&\Pi_{(1),0}+\Pi_{(1),1}+\dots+\Pi_{(1),\ell-1}\\
=\,&(1\otimes \pzs_{(1)}\otimes 1^{\otimes \ell-1}+2 \pzP_{(1)}\otimes  1^{\otimes \ell} )
+(1\otimes \pzs_{(1)}\otimes 1^{\otimes \ell-1}+1\otimes 1\otimes \pzs_{(1)}\otimes 1^{\otimes \ell-2})
\\&+\dots+(1^{\otimes \ell-1}\otimes \pzs_{(1)}\otimes 1+1^{\otimes \ell-1}\otimes 1\otimes \pzs_{(1)})
\\
=\,&\pzq_1\otimes 1^{\otimes \ell}+1^{\otimes \ell}\otimes \pzs_{(1)}+
2\sum_{i=1}^{\ell-1}1^{\otimes i}\otimes \pzs_{(1)}\otimes 1^{\otimes\ell-i},
\end{align*}
as required.

For the inductive step, for $d>1$, it follows from  Lemma~\ref{LPiSplit} that $\Pi_{\om_d}=\Pi_{\om_{d-1}}\Pi_{\om_1}$. So, by the inductive assumption, we get
\begin{align*}
\Pi_{\om_d}=\,&\Pi_{\om_{d-1}}\Pi_{\om_1}
\\
=\,
&\Big(\sum_{m_0+m_1+\dots+m_\ell=d-1}2^{m_1+\dots+m_{\ell-1}}{{d-1}\choose{m_0\,m_1\,\cdots\,m_\ell}}\pzq_1^{m_0}\otimes \pzs_{(1)}^{m_1}\otimes\dots\otimes \pzs_{(1)}^{m_\ell}\,\Big)
\\&\times \Big(\sum_{n_0+n_1+\dots+n_\ell=1}2^{n_1+\dots+n_{\ell-1}}\pzq_1^{n_0}\otimes \pzs_{(1)}^{n_1}\otimes\dots\otimes \pzs_{(1)}^{n_\ell}\,\Big)
\\
=\,&\sum_{k_0+k_1+\dots+k_\ell=d}2^{k_1+\dots+k_{\ell-1}}{{d}\choose{k_0\,k_1\,\cdots\,k_\ell}}\pzq_1^{k_0}\otimes \pzs_{(1)}^{k_1}\otimes\dots\otimes \pzs_{(1)}^{k_\ell}
\end{align*}
thanks to the identity ${{d}\choose{k_0\,k_1\,\cdots\,k_\ell}}=\sum_{r\,\,\text{with}\,\,k_r>0}{{d-1}\choose{k_0\,\cdots\, k_{r-1}\,\cdots\,k_r-1\,k_{r+1}\,\cdots k_\ell}}$.
\end{proof}

\subsection{Another description of $\Pi_{\mu,\bj}$}
Let $M_{n,I}$ denote the set of $n\times I$-matrices with non-negative integer entries, i.e.
$$
M_{n,I}=\{(a_{r,i})_{1\leq r\leq n,\, i\in I}\mid a_{r,i}\in\Z_{\geq 0}\}.
$$
For $(\mu,\bj)\in \La^\col(n,d)$, we define the sets of matrices 
\begin{align*}
M_{n,I}(\mu)&:=\{(a_{r,i})\in M_{n,I}\mid \textstyle\sum_{i\in I}a_{r,i}=\mu_r\ \text{for $r=1,\dots,n$}\},
\\
M_{n,I}(\bj)&:=\{(a_{r,i})\in M_{n,I}\mid \text{$a_{r,i}=0$ if $i\neq j_r,j_{r}+1$ for $r=1,\dots, n$}\},
\\
M(\mu,\bj)&:=M_{n,I}(\mu)\cap M_{n,I}(\bj).
\end{align*}
Let $A= (a_{r,i})\in M(\mu,\bj)$. For $1\leq r\leq n$ and $i\in I\setminus\{0\}$, we define
$$
\psi_A(r,i):=
\left\{
\begin{array}{ll}
\pzh_{a_{r,i}} &\hbox{if $i=j_r$,}\\
\pze_{a_{r,i}} &\hbox{if $i=j_{r}+1$,}\\
1 &\hbox{otherwise.}
\end{array}
\right.
$$
We now set 
\begin{align*}
\psi_A^{(0)}:=\pzq_{a_{1,0}}\cdots \pzq_{a_{n,0}},
\qquad
\psi_A^{(i)}:=\psi_A(1,i)\cdots \psi_A(n,i)
\quad(\text{for}\ i\in I\setminus\{0\}),
\end{align*}
and
$$
\psi_A:=\psi_A^{(0)}\otimes \psi_A^{(1)}\otimes\dots\otimes \psi_A^{(\ell)}\in\Sym^I.
$$

\begin{Example} \label{Exn=1}
{\rm 
Suppose $n=1$ and so $(\mu,\bj)\in\La^\col(1,d)$ is of the form $((d),j)$. The set $M((d),j)$ consists of all matrices of the form
$$
\{A(j,k):=\left(
\begin{matrix}
0 & \cdots &0 &k&d-k&0&\cdots & 0
\end{matrix}
\right)\mid 0\leq k\leq d\}
$$
with $k$ in position $j$. Note that by definition we have 
\begin{align*}
\psi_{A_{0,0}}&=\pzq_0\otimes \pze_d\otimes 1^{\otimes \ell-1}=1\otimes \pze_d\otimes 1^{\otimes \ell-1},\\
\psi_{A_{0,k}}&=\pzq_k\otimes \pze_{d-k}\otimes 1^{\otimes \ell-1}=2\pzP_{(k)}\otimes \pze_{d-k}\otimes 1^{\otimes \ell-1}\qquad(1\leq k\leq d),\\
\psi_{A_{j,k}}&=1^{\otimes j}\otimes \pzh_{k}\otimes \pze_{d-k}\otimes 1^{\otimes \ell-1-j}
\qquad(1\leq j<\ell).
\end{align*}
In particular, comparing with (\ref{EPiMuBjn=1}), we deduce that
$\Pi_{(d),j}=\sum_{A\in M((d),j)}\psi_A$. 
}
\end{Example}

\begin{Proposition} \label{PMatrices}
Let $(\mu,\bj)\in \La^\col(n,d)$. Then
$$
\Pi_{\mu,\bj}=\sum_{A\in M(\mu,\bj)}\psi_A.
$$
\end{Proposition}
\begin{proof}
We apply induction on $n$. For the base $n=1$, see Example~\ref{Exn=1}. Suppose $n\geq 2$ and set 
$$\nu:=(\mu_1,\dots,\mu_{n-1}),\quad \bk:=(j_1,\dots,j_{n-1}),\quad m:=\mu_n,\quad j:=j_n.
$$ 
Then 
$\Pi(\mu,\bj)=\Pi(\nu,\bk)\Pi((m),j)$ by Lemma~\ref{LPiSplit}. By the inductive assumption, we have 
$$
\Pi(\nu,\bk)=\sum_{B\in M(\nu,\bk)}\psi_B
\quad\text{and}\quad
\Pi((m),j)=\sum_{C\in M((m),j)}\psi_C,
$$
so it suffices to observe  that 
$$\Big(\sum_{B\in M(\nu,\bk)}\psi_B
\Big)\Big(\sum_{C\in M((m),j)}\psi_C\Big)=\sum_{A\in M(\mu,\bj)}\psi_A,$$
which comes from the definitions.  
\end{proof}

\subsection{Computing the inner product $(\Pi_{\mu,\bj},\Pi_{\om_d})_\Sym$}
Recall the inner product $(\cdot,\cdot)_\Sym$ from  (\ref{ESymInnerProd}). 
Throughout this subsection, we fix $(\mu,\bj)\in\La^\col(n,d)$. For  $A= (a_{r,i})\in M(\mu,\bj)$ and $i\in I$, we define 
$$
|a_{*,i}|:=\sum_{r=1}^n a_{r,i}.
$$
Then we have compositions 
$$
a_{*,i}:=(a_{1,i},\dots,a_{n,i})\in\La(n,|a_{*,i}|)\qquad(i\in I)
$$
and multinomial coefficients 
$$
\binom{|a_{*,i}|}{a_{*,i}}:=\binom{|a_{*,i}|}{a_{1,i}\,\cdots\, a_{n,i}}=\frac{|a_{*,i}|!}{a_{1,i}!\cdots a_{n,i}!}. 
$$

\begin{Lemma} \label{LIndividualInner}
Let $A\in M(\mu,\bj)$, and $k_0,k_1,\dots,k_\ell\in\Z_{\geq 0}$. Then
$$
(\psi_A,\pzq_1^{k_0}\otimes \pzs_{(1)}^{k_1}\otimes\dots\otimes \pzs_{(1)}^{k_\ell})_\Sym=
\left\{
\begin{array}{ll}
2^{|a_{*,0}|}\prod_{i\in I}\binom{|a_{*,i}|}{a_{*,i}} &\hbox{if $|a_{*,i}|=k_i$ for all $i\in I$,}\\
0 &\hbox{otherwise.}
\end{array}
\right.
$$
\end{Lemma}
\begin{proof}
We have
\begin{align*}
(\psi_A,\pzq_1^{k_0}\otimes \pzs_{(1)}^{k_1}\otimes\dots\otimes \pzs_{(1)}^{k_\ell})_\Sym&=
(\psi_A^{(0)}\otimes \psi_A^{(1)}\otimes\dots\otimes \psi_A^{(\ell)},\pzq_1^{k_0}\otimes \pzs_{(1)}^{k_1}\otimes\dots\otimes \pzs_{(1)}^{k_\ell})_\Sym
\\
&=[\psi_A^{(0)},\pzq_1^{k_0}]\,\,(\psi_A^{(1)},\pzs_{(1)}^{k_1})\,\cdots \,(\psi_A^{(\ell)},\pzs_{(1)}^{k_\ell}).
\end{align*}
Now, by (\ref{EQQ}), we have 
$$
[\psi_A^{(0)},\pzq_1^{k_0}]=[\pzq_{a_{1,0}}\cdots \pzq_{a_{n,0}},\pzq_1^{k_0}]
=
\left\{
\begin{array}{ll}
2^{k_0}\binom{k_0}{a_{1,0}\,\cdots\, a_{n,0}}&\hbox{if $k_0=a_{1,0}+\cdots+ a_{n,0}$,}\\
0 &\hbox{otherwise.}
\end{array}
\right.
$$
On the other hand, by (\ref{EPsiReg}), we have for $i=1,\dots,\ell$:
$$
(\psi_A^{(i)},\pzs_{(1)}^{k_i}))
=(\psi_A(1,i)\cdots\psi_A(n,i),\pzs_{(1)}^{k_i}))
\left\{
\begin{array}{ll}
\binom{k_i}{a_{1,i}\,\cdots\, a_{n,i}}&\hbox{if $k_i=a_{1,i}+\cdots+ a_{n,i}$,}\\
0 &\hbox{otherwise.}
\end{array}
\right.
$$
This implies the required equality. 
\end{proof}

Recall the notation $|\mu,\bj|_{\ell-1}$ from (\ref{EmMuj}). 

\begin{Theorem} \label{TMainInner}
Let $(\mu,\bj)\in\La^\col(n,d)$. Then 
$$
\big(\Pi_{\mu,\bj},\Pi_{\om_d}\big)_\Sym=\binom{d}{\mu_1\,\cdots\,\mu_n}\,4^{d-|\mu,\bj|_{\ell-1}}\,3^{|\mu,\bj|_{\ell-1}}.
$$
\end{Theorem}
\begin{proof}
In view of 
Proposition~\ref{PMatrices} and 
Corollary~\ref{C1^d}, we have that $\big(\Pi_{\mu,\bj},\Pi_{\om_d}\big)_\Sym$ equals
\begin{align*}
&\sum_{A\in M(\mu,\bj)}\sum_{k_0+k_1+\dots+k_\ell=d}2^{k_1+\dots+k_{\ell-1}}{{d}\choose{k_0\,k_1\,\cdots\,k_\ell}}\big(\psi_A,\pzq_1^{k_0}\otimes \pzs_{(1)}^{k_1}\otimes\dots\otimes \pzs_{(1)}^{k_\ell}\big)_\Sym.
\end{align*}
By Lemma~\ref{LIndividualInner}, this equals 
\begin{align*}
&\sum_{A\in M(\mu,\bj)}2^{|a_{*,1}|+\dots+|a_{*,\ell-1}|}{{d}\choose{|a_{*,0}|\,|a_{*,1}|\,\cdots\,|a_{*,\ell}|}}2^{|a_{*,0}|}\prod_{i\in I}\binom{|a_{*,i}|}{a_{*,i}}
\\
=\,&\sum_{A\in M(\mu,\bj)}2^{|a_{*,0}|+|a_{*,1}|+\dots+|a_{*,\ell-1}|}\frac{d!}{\prod_{i\in I}\prod_{r=1}^na_{r,i}!}
\\
=\,&\sum_{A\in M(\mu,\bj)}2^{d-|a_{*,\ell}|}\binom{d}{\mu_1\,\cdots\,\mu_n}
\prod_{r=1}^n\binom{\mu_r}{a_{r,0}\,\cdots a_{r,\ell}},
\end{align*}
and it remains to prove that
\begin{equation}\label{EToProve}
\sum_{A\in M(\mu,\bj)}2^{d-|a_{*,\ell}|}\prod_{r=1}^n\binom{\mu_r}{a_{r,0}\,\cdots a_{r,\ell}}=4^{d-|\mu,\bj|_{\ell-1}}\,3^{|\mu,\bj|_{\ell-1}}.
\end{equation}

Denote $$
d_{j}:=\sum_{\substack{1\leq r\leq n,\\ j_r=j}}\mu_r\qquad(j\in J).
$$
In particular, $d_{\ell-1}=|\mu,\bj|_{\ell-1}$ and $d_0+d_1+\dots+d_{\ell-1}=d$. 
Note that permuting the parts of $(\mu_1,\dots,\mu_n)$ and $(j_1,\dots,j_n)$ by the same permutation in $\Si_n$ does not change the left hand side of (\ref{EToProve}), so we may assume without loss of generality that $\bj=(0^{n_0},1^{n_1},\dots,(\ell-1)^{n_{\ell-1}})$ with $n_0+n_1+\dots+n_{\ell-1}=n$ and 
$$
\mu=(\la^{(0)}_1,\dots,\la^{(0)}_{n_0},\dots,\la^{(\ell-1)}_1,\dots,\la^{(\ell-1)}_{n_{\ell-1}})
$$ 
with $(\la^{(j)}_1,\dots,\la^{(j)}_{n_j})\in\La(n_j,d_j)$ for all $j\in J$. Then, the matrices $A\in M(\mu,\bj)$ look like 
$$
A=
\left(
\begin{matrix}
 B^{(0)}   \\
 \vdots \\
 B^{(\ell-1)} 
\end{matrix}
\right)
$$
where, for each $j$, the matrix $B^{(j)}=(b^{(j)}_{r,i})_{1\leq r\leq n_j,\,i\in I}$ is an arbitrary matrix with non-negative integer values with the property that $b^{(j)}_{r,i}=0$ unless $i\in\{j,j+1\}$ and $b^{(j)}_{r,j}+b^{(j)}_{r,j+1}=\la^{(j)}_r$ for all $r=1,\dots,n_j$. 
So, the left hand side of (\ref{EToProve}) 
equals $XY$ where
$$
X:=\prod_{j=0}^{\ell-2}\prod_{r=1}^{n_j}\sum_{b^{(j)}_{r,j}+b^{(j)}_{r,j+1}=\la^{(j)}_r}2^{\la^{(j)}_r}\binom{\la^{(j)}_r}{b^{(j)}_{r,j}\,b^{(j)}_{r,j+1}}
$$
and 
$$
Y:=\prod_{r=1}^{n_{\ell-1}}\sum_{b^{(\ell-1)}_{r,\ell-1}+b^{(\ell-1)}_{r,\ell}=\la^{(\ell-1)}_r}2^{b^{(\ell-1)}_{r,\ell-1}}\binom{\la^{(\ell-1)}_r}{b^{(\ell-1)}_{r,\ell-1}\,b^{(\ell-1)}_{r,\ell}}.
$$
Now, using the formula 
$\sum_{a+b=c}\binom{c}{a\,b}=2^c$, 
we get  
$$
X=2^{d_0+\dots+d_{\ell-2}}
\prod_{j=0}^{\ell-2}\prod_{r=1}^{n_j}\sum_{b^{(j)}_{r,j}+b^{(j)}_{r,j+1}=\la^{(j)}_r}2^{\la^{(j)}_r}
=4^{d_0+\dots+d_{\ell-2}}
=4^{d-d_{\ell-1}}
=4^{d-|\mu,\bj|_{\ell-1}},
$$
and, using the formula 
$\sum_{a+b=c}2^a\binom{c}{a\,b}=3^c$, 
we get  
$$
Y=\prod_{r=1}^{n_{\ell-1}}3^{\la^{(\ell-1)}_r}=3^{d_{\ell-1}}=3^{|\mu,\bj|_{\ell-1}},
$$
completing the proof.
\end{proof}

\section{Fock space}
\label{SFock}

\subsection{Fock spaces $\Fock_q$ and $\Fock$}
The  {\em ($q$-deformed) level\, $1$ Fock space} $\Fock_q$, as defined in \cite{KMPY} (see also \cite{LT}), is the $\Q(q)$ vector space with basis $\{u_{\la}\mid \la\in \Par_p\}$ labeled by the $p$-strict partitions:
$$
\Fock_q:=\bigoplus_{\la\in\Par_p} \Q(q)\cdot  u_{\la}.
$$

There is a structure of a $U_q(\g)$-module on $\Fock_q$ such that 
\begin{eqnarray}\label{EActionFockE}
E_iu_\la&=&\sum_{\ttA\in\PR_i(\la)}d_\ttA(\la)u_{\la_\ttA},\\
\label{EActionFockF}
F_iu_\la&=&\sum_{\ttB\in\PA_i(\la)}d^\ttB(\la)u_{\la^\ttB},
\\
\label{EActionFockT}
T_iu_\la&=&
q^{(\al_i|\La_0-\cont(\la))}u_\la.
\end{eqnarray}

\begin{Example} 
{\rm 
In the set up of Example~\ref{Ex1} we have 
$$F_0u_{(5,5,2)}=(1-q^4)u_{(6,5,2)}+u_{(5,5,2,1)}.$$
}
\end{Example}

Moreover, as established in \cite[Appendix D]{KMPY}, there is a bilinear form $(\cdot,\cdot)$\index{$(\cdot,\cdot)_q$} on $\Fock_q$ which satisfies
\begin{equation}\label{EOrth}
(u_\la,u_\mu)_q=\de_{\la,\mu}\norm{\la}_q
\end{equation}
and 
$$(xv,w)_q=(v,\sigma_q(x)w)_q$$ for all $x\in U_q(\g)$ and $v,w\in\Fock_q$. The following well-known result allows us to identify $V_q(\La_0)$ with the submodule of $\Fock_q$ generated by $u_\varnothing$, where $\varnothing$ stands for the partition $(0)$ of $0$, cf. \cite[Lemma 2.4.20]{KlLi}.

\begin{Lemma} \label{LForm}  
There is a unique isomorphism of $U_q(\g)$-modules $V_q(\La_0)\iso U_q(\g)\cdot u_\varnothing$ mapping $ v_{+,q}$ onto $u_\varnothing$.  Moreover, identifying $V_q(\La_0)$ with the submodule $U_q(\g)\cdot u_{\varnothing}\subseteq \Fock_q$ via this isomorphism, the Shapovalov form   $(\cdot,\cdot)_q$ on $V_q(\La_0)$ is the restriction of the form $(\cdot,\cdot)_q$ on $\Fock_q$ to $V_q(\La_0)$.
\end{Lemma}

We now apply the construction of \S\ref{SSQ1} to go from the $U_q(\g)$-module $\Fock_q$ to the $\g$-module $\Fock_q\mid_{q=1}=\C\otimes_{\Z[q,q^{-1}]}\Fock_{q,\Z[q,q^{-1}]}$ with the lattice $\Fock_{q,\Z[q,q^{-1}]}:=\bigoplus_{\la\in\Par_p} \Z[q,q^{-1}]\cdot u_\la$. We will denote $1\otimes u_\la \in \C\otimes_{\Z[q,q^{-1}]}\Fock_{q,\Z[q,q^{-1}]}$ again by $u_\la$. So we have a $\g$-module
$$
\Fock_q\mid_{q=1}=\bigoplus_{\la\in\Par_p}\C\cdot u_\la
$$
with the action 
\begin{eqnarray}\label{EActionFockEDeg}
e_iu_\la=\sum_{\ttA\in\PR_i(\la)}(d_\ttA(\la)\mid_{q=1})u_{\la_\ttA},\quad
f_iu_\la=\sum_{\ttB\in\PA_i(\la)}(d^\ttB(\la)\mid_{q=1})u_{\la^\ttB}, 
\end{eqnarray}
and the form $(\cdot,\cdot):=(\cdot,\cdot)_q\mid_{q=1}$ which satisfies  
$$
(u_\la,u_\mu)=\de_{\la,\mu}\norm{\la}_q\mid_{q=1} 
$$
and $(xv,w)=(v,\sigma(x)w)$ for all $x\in \g$ and $v,w\in\Fock_q\mid_{q=1}$.

Recalling (\ref{EDQE1}) and (\ref{EDQF1}), it is easy to see that  ${\mathscr R}:=\spa_\C(u_\la\mid\la\in\Par_p\setminus\Par_0)$ is a $\g$-submodule of $\Fock_q\mid_{q=1}$. Consider the {\em reduced Fock space}
$$
\Fock:=(\Fock_q\mid_{q=1})/{\mathscr R}.
$$
Denoting $u_\la+{\mathscr R}\in (\Fock_q\mid_{q=1})/{\mathscr R}$ by $u_\la$ again, we have that 
$$
\Fock=\bigoplus_{\la\in\Par_0}\C\cdot u_\la
$$
with the action of the Chevalley generators $f_i$ given by: 
\begin{eqnarray}\label{EActionFockEDegRed}
f_iu_\la=\sum_{\ttB\in A_i(\la)}c(\la,\ttB)u_{\la^\ttB},
\end{eqnarray}
where 
$$
A_i(\la):=\{\ttB\in \PA_i(\la)\mid \la^\ttB\in\Par_0\}
$$
and, recalling (\ref{Ehp}), 
\begin{equation}\label{EclaB}
c(\la,\ttB):=
\left\{
\begin{array}{ll}
2 &\hbox{if $h_p(\la^\ttB)=h_p(\la)-1$,}\\
1 &\hbox{otherwise.}
\end{array}
\right.
\end{equation}
(We are not going to need the action of the Chevalley generators $e_i$.)
Moreover, $\Fock$ inherits the form  $(\cdot,\cdot)$ which satisfies  
\begin{equation}\label{EFormDegFock}
(u_\la,u_\mu)=\de_{\la,\mu}2^{h_p(\la)}. 
\end{equation}
and $(xv,w)=(v,\sigma(x)w)$ for all $x\in \g$ and $v,w\in\Fock$. We now have from Lemmas~\ref{LForm} and \ref{LShapqShap}:

\begin{Lemma} \label{LFormDeg} 
There is a unique isomorphism of $\g$-modules from $V(\La_0)$ to the submodule $U(\g)\cdot u_{\varnothing}\subseteq \Fock$  generated by $ u_\varnothing$,  mapping $ v_{+}$ onto $u_\varnothing$.  Moreover, identifying $V(\La_0)$ with $U(\g)\cdot u_{\varnothing}\subseteq \Fock$ via this isomorphism, the Shapovalov form   $(\cdot,\cdot)$ on $V(\La_0)$ is the restriction of the form $(\cdot,\cdot)$ on $\Fock$ to $V(\La_0)$.
\end{Lemma}

Let $\rho$ be a $\bar p$-core. By definition, we have $h_p(\rho)=0$ and $\rho\in\Par_0$. Note that by Lemma~\ref{LCorew}, there is $w\in W$ such that $\cont(\rho)=\La_0-w\La_0$. We have the element $v_w\in V(\La_0)$ defined in Lemma~\ref{LWExt}, and the element $u_\rho\in\Fock$. The following lemma shows that these agree:

\begin{Lemma} \label{LCorevw}
Let $\iota:V(\La_0)\iso U(\g)\cdot u_{\varnothing},\ v_{+}\mapsto u_\varnothing$ be the isomorphism of Lemma~\ref{LFormDeg}. 
If $\rho$ is a $\bar p$-core and $w\in W$ is such that $\cont(\rho)=\La_0-w\La_0$ then $\iota(v_w)=u_\rho$. 
\end{Lemma}
\begin{proof}
By Lemma~\ref{LWExt}, for certain $a_1,\dots,a_l$, we have 
$$\iota(v_w)=\iota(F_{i_l}^{(a_l)}\cdots F_{i_1}^{(a_1)}v_+)
=F_{i_l}^{(a_l)}\cdots F_{i_1}^{(a_1)}\iota(v_+)=F_{i_l}^{(a_l)}\cdots F_{i_1}^{(a_1)}u_\varnothing.
$$
It follows from Lemma~\ref{LCorew}, that $\Fock_{w\La_0}=\Fock_{\La_0-\cont(\rho)}$ is $1$-dimensional and hence spanned by $u_\rho$. Now, it follows from the formulas (\ref{EActionFockEDegRed}) and (\ref{EclaB}) that $F_{i_l}^{(a_l)}\cdots F_{i_1}^{(a_1)}u_\varnothing=ku_\rho$ for $k\in\Z_{>0}$. Now, $(u_\rho,u_\rho)=2^{h_p(\rho)}=1$. On the other hand, 
$$(F_{i_l}^{(a_l)}\cdots F_{i_1}^{(a_1)}u_\varnothing,F_{i_l}^{(a_l)}\cdots F_{i_1}^{(a_1)}u_\varnothing)=(\iota(v_w),\iota(v_w))=(v_w,v_w)=1
$$ 
by Lemmas~\ref{LFormDeg} and \ref{LWExt}(iii). So $k=1$. 
\end{proof}

\subsection{The elements $\chi_\la$}
In this subsection we introduce some new basis of $\Fock$. 
Recalling (\ref{Ehp}), (\ref{EAla}) and (\ref{Emneq0}), we consider the following rescalings of the basis vectors $u_\la$:
\begin{equation}\label{EChi}
\chi_\la:=2^{(\sfp_\la-h_p(\la)-c^\la_{\neq 0})/2}u_\la\qquad(\la\in\Par_0).
\end{equation}
These elements correspond to the irreducible supercharacters of the double covers of symmetric groups, see \cite[\S5]{FKM}, and the formula (\ref{EChi}) was communicated to us by M. Fayers. 

Set
$$
a(\la,\ttB):=
\left\{
\begin{array}{ll}
2 &\hbox{if $\sfp_\la=1$ and $\sfp_{\la^\ttB}=0$,}\\
1 &\hbox{otherwise.}
\end{array}
\right.
$$

\begin{Lemma}\label{LFChi}
Let $i\in I$ and $\la\in\Par_0$. Then 
$f_i\chi_\la=\sum_{\ttB\in A_i(\la)}a(\la,\ttB)u_{\la^\ttB}.$
\end{Lemma}
\begin{proof}
We have 
\begin{align*}
f_i\chi_\la&=2^{(\sfp_\la-h_p(\la)-c^\la_{\neq 0})/2}f_iu_\la
\\
&=2^{(\sfp_\la-h_p(\la)-c^\la_{\neq 0})/2}\sum_{\ttB\in A_i(\la)}c(\la,\ttB)u_{\la^\ttB}
\\
&=2^{(\sfp_\la-h_p(\la)-c^\la_{\neq 0})/2}\sum_{\ttB\in A_i(\la)}c(\la,\ttB)2^{-(\sfp_{\la^\ttB}-h_p(\la^\ttB)-c^{\la^\ttB}_{\neq 0})/2}\chi_{\la^\ttB},
\end{align*}
so we need to prove that 
\begin{equation}\label{EReq}
a(\la,\ttB)=2^{(\sfp_\la-\sfp_{\la^\ttB}-h_p(\la)+h_p(\la^\ttB)-c^\la_{\neq 0}+c^{\la^\ttB}_{\neq 0})/2}c(\la,\ttB).
\end{equation}

If $i\neq 0$ then either $\sfp_\la=1$ and $\sfp_{\la^\ttB}=0$, or $\sfp_\la=0$ and $\sfp_{\la^\ttB}=1$. Moreover, $c^{\la^\ttB}_{\neq 0}=c^{\la}_{\neq 0}+1$. Hence 
\begin{align*}
2^{(\sfp_\la-\sfp_{\la^\ttB}-h_p(\la)+h_p(\la^\ttB)-c^\la_{\neq 0}+c^{\la^\ttB}_{\neq 0})/2}c(\la,\ttB)&=
2^{(\sfp_\la-\sfp_{\la^\ttB}+1)/2}
\\
&=\left\{
\begin{array}{ll}
2 &\hbox{if $\sfp_\la=1$ and $\sfp_{\la^\ttB}=0$,}\\
1 &\hbox{otherwise.}
\end{array}
\right.
\end{align*}
which is $a(\la,B)$ as required. 

On the other hand, if $i=0$ then $c^\la_{\neq 0}=c^{\la^\ttB}_{\neq 0}$.  Let $\la$ be written in the form (\ref{EStrictForm}), and consider the following three cases:

\vspace{1mm}
(1) $B=(m_1+\dots+m_k+1,1)$. In this case, we have $c(\la,B)=1$,  $\sfp_\la=\sfp_{\la^\ttB}$ and $h_p(\la)=h_p(\la^\ttB)$, which immediately gives (\ref{EReq}). 

\vspace{1mm}
(2) $B=(m_1+\dots+m_{r-1}+1,1)$ for some $1\leq r\leq k$ and $p\mid l_r$. In this case, we have $h_p(\la^\ttB)=h_p(\la)-1$, 
$c(\la,B)=2$ and $\sfp_{\la^\ttB}\neq \sfp_\la$. If $\sfp_{\la^\ttB}=1$ and $\sfp_\la=0$ then both sides of (\ref{EReq}) equal $1$. 
If $\sfp_{\la^\ttB}=0$ and $\sfp_\la=1$ then both sides of (\ref{EReq}) equal $2$. 

\vspace{1mm}
(3) $B=(m_1+\dots+m_{r-1}+1,1)$ for some $1\leq r\leq k$ and $p\nmid l_r$. 
In this case, we have $h_p(\la^\ttB)=h_p(\la)+1$, 
$c(\la,B)=1$ and $\sfp_{\la^\ttB}\neq \sfp_\la$. If $\sfp_{\la^\ttB}=1$ and $\sfp_\la=0$ then both sides of (\ref{EReq}) equal $1$. 
If $\sfp_{\la^\ttB}=0$ and $\sfp_\la=1$ then both sides of (\ref{EReq}) equal $2$. 
\end{proof}

\section{Shapovalov form for RoCK weights}
\label{S4}

Suppose that $\theta\in Q_+$ satisfies $V(\La_0)_{\La_0-\theta}\neq 0$. 
Then $\theta=\La_0-w\La_0+d\de$ for some $w$ in the Weyl group $W$, and unique $d\in \Z_{\geq 0}$. 
We say that $\theta$ is {\em RoCK}\, if $(\theta\mid \al_0^\vee)\geq 2d$ and $(\theta\mid \al_i^\vee)\geq d-1$ for $i=1,\dots,\ell$. 
This is equivalent to the cyclotomic quiver Hecke superalgebra $R_\theta^{\La_0}$ being a {\em RoCK block}, as defined in \cite[Section 4.1]{KlLi}. 

Throughout Section~\ref{S4}, we fix a RoCK weight $\theta\in Q_+$, so that $\theta=\La_0-w\La_0+d\de$ for some $w\in W$ and $d\in \Z_{\geq 0}$, and $(\theta\mid \al_0^\vee)\geq 2d$, $(\theta\mid \al_i^\vee)\geq d-1$ for $i=1,\dots,\ell$. 

We have the element $v_w\in V(\La_0)_{w\La_0}$ defined in Lemma~\ref{LWExt}. 

\subsection{Computation of $f(\mu,\bj)u_\rho$}
For each $m\in \Z_{\geq 0}$, $j\in J$, $(\mu,\bj)\in \La^\col(n,d)$,  recall the divided power monomials $f(m,j)$ and 
$f(\mu,\bj)$ defined in (\ref{EFmj}) and (\ref{EFmubj}). We also have a sum $f(\om_d)$ of monomials  defined in (\ref{Efom}). 

Let $\mu\in\Par_0(\rho,c)$ with $c\leq d$. Recall the notion of a $p$-quotient $\quot(\mu)=(\mu^{(0)},\dots,\mu^{(\ell)})$ of $\mu$ from (\ref{EPQuot}) and the notation (\ref{EqAlLa}). Recall that $\quot(\mu)\in\Par^I_0(c)$ since $\mu$ is strict. The following result follows immediately from \cite[Proposition 6.6]{FKM}, \cite[(5.1)]{FKM}, and Lemma~\ref{LFChi}. 

\begin{Lemma} \label{LMatt} 
Let $j\in J$ and $c,k\in\Z_{\geq 0}$ satisfy $c+k\leq d$. For $\al\in\Par_0(\rho,c)$, 
 in the reduced Fock space $\Fock$, we have 
$$
f(k,j)\chi_\al=\sum_\la 2^{q(\la^{(0)}/\al^{(0)})+(k(2\ell-1)+h(\al^{(0)})-h(\la^{(0)})+\sfp_\al-\sfp_\la)/2}\chi_\la
$$
where the sum is over all $\la\in\Par_0(\rho,c+k)$ such that 
$\quot(\la)=(\la^{(0)},\dots,\la^{(\ell)})$ is obtained from $\quot(\al)=(\al^{(0)},\dots,\al^{(\ell)})$ 
by adding $k$ nodes to the components $\al^{(j)}$ and $\al^{(j+1)}$, with no two nodes added in the same column of $\al^{(j)}$ or in the same row of $\al^{(j+1)}$.
\end{Lemma}

\begin{Corollary} \label{CMatt} 
Let $j\in J$ and $c,k\in\Z_{\geq 0}$ satisfy $c+k\leq d$. For $\al\in\Par_0(\rho,c)$, 
 in the reduced Fock space $\Fock$, we have 
$$
f(k,j)u_\al=\sum_\la 2^{q(\la^{(0)}/\al^{(0)})+h(\al^{(0)})-h(\la^{(0)})}u_\la
$$
where the sum is over all $\la\in\Par_0(\rho,c+k)$ such that 
$\quot(\la)=(\la^{(0)},\dots,\la^{(\ell)})$ is obtained from $\quot(\al)=(\al^{(0)},\dots,\al^{(\ell)})$ 
by adding $k$ nodes to the components $\al^{(j)}$ and $\al^{(j+1)}$, with no two nodes added in the same column of $\al^{(j)}$ or in the same row of $\al^{(j+1)}$.
\end{Corollary}
\begin{proof}
Note that $h_p(\al)=h(\al^{(0)})$. 
Moreover, for $\la$'s appearing in the sum, we have $c^\la_{\neq 0}-c^\al_{\neq 0}=k(2\ell-1)$ and $h_p(\la)=h(\la^{(0)})$. 
So we have by (\ref{EChi}) and Lemma~\ref{LMatt},
\begin{align*}
f(k,j)u_\al&=
2^{(-\sfp_\al+h_p(\al)+c^\al_{\neq 0})/2}
f(k,j)\chi_\al
\\
&=
2^{(-\sfp_\al+h_p(\al)+c^\al_{\neq 0})/2}\sum_\la 2^{q(\la^{(0)}/\al^{(0)})+(k(2\ell-1)+h(\al^{(0)})-h(\la^{(0)})+\sfp_\al-\sfp_\la)/2}\chi_\la
\\
&=
2^{q(\la^{(0)}/\al^{(0)})+h(\al^{(0)})+(k(2\ell-1)-h(\la^{(0)})-\sfp_\la+c^\al_{\neq 0})/2}\sum_\la 
2^{(\sfp_\la-h_p(\la)-c^\la_{\neq 0})/2}u_\la
\\
&=
\sum_\la 
2^{q(\la^{(0)}/\al^{(0)})+h(\al^{(0)})-h(\la^{(0)})}
u_\la,
\end{align*}
as required.
\end{proof}

\begin{Lemma} \label{LFMuBjuRho}
Let $(\mu,\bj)\in \La^\col(n,d)$. Then 
$$
f(\mu,\bj) u_\rho=\sum_{\la\in\Par_0(\rho,d)}K(\quot(\la);\mu,\bj)2^{-h(\la^{(0)})}u_\la.
$$
\end{Lemma}
\begin{proof}
We apply induction on $n$. For the induction base case $n=1$ we apply Corollary~\ref{CMatt} to see that 
$$
f(\mu_1,j_1)u_\rho=\sum_\la 2^{q(\la^{(0)}/\rho^{(0)})+h(\rho^{(0)})-h(\la^{(0)})}
u_\la=\sum_\la 2^{q(\la^{(0)}/\rho^{(0)})}2^{-h(\la^{(0)})}
u_\la,
$$
where where the sums are over all $\la\in\Par_0(\rho,d)$ such that 
$\quot(\la)=(\la^{(0)},\dots,\la^{(\ell)})$ is obtained from $\quot(\rho)=(\varnothing,\dots,\varnothing)$ 
by adding $k$ nodes to the components $\rho^{(j)}=\varnothing$ and $\rho^{(j+1)}=\varnothing$, with no two nodes added in the same column of $\rho^{(j)}$ or in the same row of $\rho^{(j+1)}$, and $$
q(\la^{(0)}/\rho^{(0)})=|\{r\in\Z_{>0}\mid \la^{(0)}\ \text{contains a node in column $r$ but not in column $r+1$}\}|.
$$
It follows from the definitions that the $\la$'s appearing in the latter sum are exactly the $\la$'s with $K(\quot(\la);(\mu_1),j_1)\neq 0$, and for those $\la$ we have $K(\quot(\la);(\mu_1),j_1)=2^{q(\la^{(0)}/\rho^{(0)})}$. This establishes the induction base.

For the inductive step, suppose $n>1$. Let $\nu:=(\mu_1,\dots,\mu_{n-1})$, $\bk:=(j_1,\dots,j_{n-1})$, $c=\mu_1+\dots+\mu_{n-1}$. In particular, $f(\mu,\bj)=f(\mu_n,j_n)f(\nu,\bk)$. By the inductive assumption, we have
\begin{align*}
f(\mu,\bj) u_\rho&=f(\mu_n,j_n)f(\nu,\bk) u_\rho
\\
&=\sum_{\al\in\Par_0(\rho,c)}K(\quot(\al);\nu,\bk)2^{-h(\al^{(0)})}f(\mu_n,j_n)u_\al
\\
&=\sum_{\al\in\Par_0(\rho,c)}K(\quot(\al);\nu,\bk)2^{-h(\al^{(0)})}
\sum_\la 2^{q(\la^{(0)}/\al^{(0)})+h(\al^{(0)})-h(\la^{(0)})}u_\la
\end{align*}
where the second sum is over all $\la\in\Par_0(\rho,d)$ such that 
$\quot(\la)=(\la^{(0)},\dots,\la^{(\ell)})$ is obtained from $\quot(\al)=(\al^{(0)},\dots,\al^{(\ell)})$ 
by adding $\mu_n$ nodes to the components $\al^{(j_n)}$ and $\al^{(j_n+1)}$, with no two nodes added in the same column of $\al^{(j_n)}$ or in the same row of $\al^{(j_n+1)}$. 
It remains to note that $\la$'s appearing in the expression above are exactly those with $K(\quot(\la);\mu,\bj)\neq 0$, and for such $\la$ we have 
$$K(\quot(\la);\mu,\bj)=\sum_\al 2^{q(\la^{(0)}/\al^{(0)})}K(\quot(\al);\nu,\bk),$$ where the sum is over all $\al\in\Par_0(\rho,c)$ such that 
$\quot(\al)=(\al^{(0)},\dots,\al^{(\ell)})$ is obtained from $\quot(\la)=(\la^{(0)},\dots,\la^{(\ell)})$ 
by removing $\mu_n$ nodes to the components $\la^{(j_n)}$ and $\la^{(j_n+1)}$, with no two nodes removed in the same column of $\la^{(j_n)}$ or in the same row of $\la^{(j_n+1)}$. 
\end{proof}

Recall the definition $\Pi_{\mu,\bj}$ of from (\ref{EPiMuBj}). 

\begin{Corollary} \label{CFormForm}
Let $(\mu,\bj)\in\La^\col(m,d)$ and $(\nu,\bi)\in\La^\col(n,d)$. Then 
$$
(f(\mu,\bj)u_\rho, f(\nu,\bi)u_\rho)=(\Pi_{\mu,\bj},\Pi_{\nu,\bi})_\Sym.
$$
\end{Corollary}
\begin{proof}
For $\la\in\Par_0(\rho,d)$, we have $h_p(\la)=h(\la^{(0)})$. So, 
by Lemma~\ref{LFMuBjuRho} and (\ref{EFormDegFock}), taking into account the bijection (\ref{EBijection}), we have 
\begin{align*}
(f(\mu,\bj)u_\rho, f(\nu,\bi)u_\rho)&=
\sum_{\la\in\Par_0(\rho,d)}K(\quot(\la);\mu,\bj)K(\quot(\la);\nu,\bi)2^{-2h(\la^{(0)})}(u_\la,u_\la)
\\
&=
\sum_{\ula\in\Par_0^I(d)}K(\ula;\mu,\bj)K(\ula;\nu,\bi)2^{-h(\la^{(0)})}.
\end{align*}
Since the bases $\{\pi_\ula\mid \ula\in\Par^I_0\}$ from (\ref{EPi})
and 
$\{\ka_\ula\mid \ula\in\Par^I_0\}$ from (\ref{EPi}) are dual to each other with respect to the inner product $(\cdot,\cdot)_\Sym$, we have 
\begin{align*}
&\sum_{\ula\in\Par_0^I(d)}K(\ula;\mu,\bj)K(\ula;\nu,\bi)2^{-h(\la^{(0)})}
\\
=\,&
(\sum_{\ula\in\Par_0^I(d)}K(\ula;\mu,\bj)2^{-h(\la^{(0)})}\ka_\ula,
\sum_{\ula\in\Par_0^I(d)}K(\ula;\nu,\bi)\pi_\ula)_\Sym
\\
=\,&
(\sum_{\ula\in\Par_0^I(d)}K(\ula;\mu,\bj)\pi_\ula,
\sum_{\ula\in\Par_0^I(d)}K(\ula;\nu,\bi)\pi_\ula)_\Sym
\\
=\,&(\Pi_{\mu,\bj},\Pi_{\nu,\bi})_\Sym,
\end{align*}
as required.
\end{proof}

\begin{Theorem} \label{TMainHalf}
Let $(\mu,\bj)\in\La^\col(m,d)$. Then 
$$
(f(\mu,\bj)u_\rho,f(\om_d)u_\rho)=\binom{d}{\mu_1\,\cdots\,\mu_n}\,4^{d-|\mu,\bj|_{\ell-1}}\,3^{|\mu,\bj|_{\ell-1}}.
$$
\end{Theorem}
\begin{proof}
By Corollary~\ref{CFormForm}, we have 
$
(f(\mu,\bj)u_\rho,f(\om_d)u_\rho)=(\Pi_{\mu,\bj},\Pi_{\om_d})_\Sym,
$
and the theorem follows from Theorem~\ref{TMainInner}. 
\end{proof}

Recall the vector $v_w\in V(\La_0)_{w\La_0}$ defined in Lemma~\ref{LWExt}. 

\begin{Theorem} \label{TMainHalf'}
Let $(\mu,\bj)\in\La^\col(m,d)$. Then 
$$
(f(\mu,\bj)v_w,f(\om_d)v_w)=\binom{d}{\mu_1\,\cdots\,\mu_n}\,4^{d-|\mu,\bj|_{\ell-1}}\,3^{|\mu,\bj|_{\ell-1}}.
$$
\end{Theorem}
\begin{proof}
In view of Lemmas~\ref{LFormDeg} and \ref{LCorevw}, 
this follows from Theorem~\ref{TMainHalf}.
\end{proof}

\end{document}